\documentclass[final,a4paper]{amsart}

\usepackage[T1]{fontenc}
\usepackage[utf8]{inputenc}
\usepackage{amssymb}
\usepackage{url}
\usepackage{datetime}
\usepackage{hyperref}
\usepackage[backend=bibtex,
            url=false,
            isbn=false,
            backref=true,
            citestyle=alphabetic,
            bibstyle=alphabetic,
            autocite=inline,
            sorting=ydnt,]{biblatex}
\usepackage[english]{babel}

\newtheorem{thm}{Theorem}[section]
\newtheorem{lem}[thm]{Lemma}
\newtheorem{cor}[thm]{Corollary}
\newtheorem{prop}[thm]{Proposition}
\newtheorem{fct}[thm]{Fact}
\newtheorem{con}{Conjecture}
\newtheorem{qu}[con]{Question}
\theoremstyle{remark}
\newtheorem*{rem}{Remark}
\newtheorem*{rems}{Remarks}
\theoremstyle{definition}
\newtheorem*{dfn}{Definition}
\newtheorem*{clm}{Claim}
\newenvironment{clmproof}[1][\proofname]{\proof[#1]}{\endproof}
\newtheorem{ex}[thm]{Example}

\DeclareMathOperator{\rng}{{rng}}

\DeclareMathOperator{\Th}{{Th}}
\DeclareMathOperator{\tp}{{tp}}
\DeclareMathOperator{\Aut}{{Aut}}
\DeclareMathOperator{\dcl}{{dcl}}

\DeclareMathOperator{\Core}{{Core}}

\newcommand{\mon}{\mathfrak C}
\newcommand{\proves}{\mathrel{\vdash}}
\newcommand{\liff}{\mathrel{\leftrightarrow}}
\newcommand{\limplies}{\mathrel{\rightarrow}}
\newcommand{\bigland}{\bigwedge}
\newcommand{\biglor}{\bigvee}
\newcommand{\Autf}[1]{\operatorname{Aut\,f}_{#1}}

\newcommand{\G}{{\mathcal G}}
\newcommand{\Z}{{\bf Z}}
\newcommand{\N}{{\bf N}}
\newcommand{\R}{{\bf R}}
\newcommand{\KP}{{\mathit{KP}}}
\newcommand{\restr}{\mathord{\upharpoonright}}
\newcommand{\wt}[1]{\widetilde{#1}}
\newcommand{\EZ}{\mathrel{ { {\bf E}_0 } } }
\newcommand{\Er}{\mathrel{E}}
\newcommand{\Fr}{\mathrel{F}}

\begin{document}

\address{
Instytut Matematyczny, Uniwersytet Wrocławski,
pl. Grunwaldzki 2/4, 50-384 Wrocław, Poland
}

\author{Krzysztof Krupi\'{n}ski}
\email[K.\ Krupi\'{n}ski]{kkrup@math.uni.wroc.pl}
\thanks{The first author is supported by NCN grant 2012/07/B/ST1/03513}

\author{Tomasz Rzepecki}
\email[T.\ Rzepecki]{tomasz.rzepecki@math.uni.wroc.pl}
\thanks{The paper is based on the Master's thesis of the second author.}


\title{Smoothness of bounded invariant equivalence relations}

\begin{abstract}
We generalise the main theorems from the paper ``The Borel cardinality of Lascar strong types'' by I.\ Kaplan, B.\ Miller and P.\ Simon to a wider class of bounded invariant equivalence relations. We apply them to describe relationships between fundamental properties of bounded invariant equivalence relations (such as smoothness or type-definability) which also requires finding a series of counterexamples. Finally, we apply the generalisation mentioned above to prove a conjecture from a paper by the first author and J. Gismatullin, showing that the key technical assumption of the main theorem (concerning connected components in definable group extensions) from that paper is not only sufficient but also necessary to obtain the conclusion. 
\end{abstract}

\maketitle

\section{Introduction}
\subsection{Preface}This paper will concern the Borel cardinalities of bounded, invariant equivalence relations, as well as some weak analogues in an uncountable case. More precisely, we are concerned with the connection between type-definability and smoothness of these relations -- type-definable equivalence relations are always smooth (cf.\ Fact~\ref{fct:tdsmt}), while the converse is not true in general. 
We also apply this to the study of connected components in definable group extensions.

The general motivation for the use of Borel cardinality in the context of bounded invariant equivalence relations is a better understanding of ``spaces'' of strong types (i.e., ``spaces'' of classes of such relations). For a bounded type-definable equivalence relation, its set of classes, equipped with the so-called logic topology, forms a compact Hausdorff topological space. However, for relations which are only invariant, but not type-definable, the logic topology is not necessarily Hausdorff, so it is not so useful. The question arises how to measure the complexity of the spaces of classes of such relations. One of the ideas is to investigate their Borel cardinalities, which was formalised in \autocite{KPS13}, wherein the authors asked whether the Lascar strong type must be non-smooth if it is not equal to the Kim-Pillay strong type. This question was answered in the positive in \autocite{KPS13}, and in this paper, we generalise its methods to a more general class of invariant equivalence relations, and we find an important application in the context of definable group extensions.

There are four main results:
\begin{enumerate}
\item
Theorem~\ref{thm:mainA}, a technical statement showing that some invariant equivalence relations are not smooth, which is proved by a simple modification of the proof of the main result of \autocite{KMS14} (Very similar results have been since shown in \autocite{KM14} using different -- though not unrelated -- methods, although it should be noted that the preprint of the latter was circulated after the proof of Theorem~\ref{thm:mainA} presented here was found by the authors.);
\item
Theorem~\ref{thm:mainAu} (which is an uncountable analogue of Theorem~\ref{thm:mainA}) and its Corollary~\ref{cor:mainGu}; again, it is obtained by a modification of a result of \autocite{KMS14}, although in this case it is somewhat more elaborate; this theorem contains some completely new information concerning the notion of sub-Vietoris topology introduced in this paper, which is essential for the application to definable group extensions in the last section of the paper;
\item
Theorem~\ref{thm:mainB}, in which we attempt to analyse in detail the connection between smoothness, type-definability and some other properties of bounded and invariant equivalence relations, under some additional assumptions; it uses a corollary of Theorem~\ref{thm:mainA} to show that some of these properties are stronger than others, and several (original) examples to show that they are not equivalent;
\item
Theorem~\ref{thm:mainext}, which applies Corollary~\ref{cor:mainGu} along with some ideas from \autocite{GK13} and \autocite{KPS13} in the context of definable group extensions, in order to give a criterion for type-definability of subgroups of such extensions, resulting in a proof of important technical conjectures (see  Conjectures~\ref{con:gk211} and~\ref{con:gk210} in the last section) from \autocite{GK13} in Corollary~\ref{cor:conpr}; the motivation for these conjectures is recalled in the remark following them.
\end{enumerate}

The main results discussed above are included in Sections 3, 4 and 5. The second section develops the necessary framework upon which we will base the part that comes after it -- the language in which we express the sequel. In particular, we introduce  the notions of orbital and orbital on types equivalence relations, as well as the notion of a normal form.

\subsection{Conventions}
In the following, unless otherwise stated, we assume that we have a fixed complete theory $T$  with infinite models. (The theory may be multi-sorted, and it will, of course, vary in some specific examples.)

We also fix a monster model $\mon\models T$, that is, a model which is $\kappa$-saturated and strongly $\kappa$-homogeneous for $\kappa$ a sufficiently large cardinal (and whenever we say ``small'' or ``bounded'', we mean smaller than this $\kappa$). If we assume that there is a sufficiently large and strongly inaccessible cardinal $\kappa$, we can take for $\mon$ the saturated model of cardinality $\kappa$.
We say that an equivalence relation on a product of sorts of $\mon$ is bounded if its number of classes is bounded.

We assume that all parameter sets are contained in $\mon$, every model we consider is an elementary substructure of $\mon$, and every tuple is of small length.
Often, we will denote by $M$ an arbitrary, but fixed small model.

For a small set $A\subseteq \mon$, by $A$-invariant we mean $\Aut(\mon/A)$-invariant.

For simplicity, whenever we mention definable, type-definable or invariant sets, we mean that they are (unless otherwise stated) \emph{$\emptyset$-definable}, \emph{$\emptyset$-type-definable} or \emph{$\emptyset$-invariant}, respectively.

When talking about tuples of elements of $\mon$, we will often say that they are in $\mon$ (as opposed to some product of various sorts of $\mon$), without specifying the length, when it does not matter or there is no risk of confusion. Likewise, we will often write $X\subseteq \mon$ when $X$ is a subset of some product of sorts of $\mon$.

If $X$ is some $A$-invariant set (esp.\ type-definable over $A$), we will denote by $S_X(A)$ the set of complete $A$-types of elements of $X$, and similarly we will sometimes omit $X$ (or names of sorts in multi-sorted context) in $S_X(A)$, and write simply $S(A)$ instead.

Throughout the paper, formulas and types will be routinely identified with the corresponding subsets of $\mon$, as well as the corresponding subsets of type spaces (or points, in case of complete types). Similarly, invariant sets will be identified with subsets of type spaces and equivalent $L_{\infty,\omega}$ formulas. For example, if $X\subseteq \mon$ is an $A$-invariant set, then we will identify $X$ with $\biglor_{i\in I} \bigland_{j\in J} \varphi_{i,j}(x,A)$ (where $I,J$ are possibly infinite index sets and $\varphi_{i,j}$ are first order formulas) if we have 
\[
x\in X\iff\mon\models \biglor_{i\in I} \bigland_{j\in J} \varphi_{i,j}(x,A).
\]
In this case, we also associate with $X$ the subset $X_A=\{\tp(a/A)\mid a\in X\}$ of $S(A)$; when $A=\emptyset$, and there is no risk of confusion, we will sometimes simply write $X$ instead of $X_\emptyset$.

When metrics are mentioned, they are binary functions into $[0,\infty]=\R_{\geq 0}\cup \{\infty\}$ satisfying the usual axioms (coincidence axiom, symmetry and triangle inequality), but in particular, they are allowed to (and usually will) attain $\infty$.

\subsection{Preliminaries}
It is assumed that the reader is familiar with basic concepts of model theory (e.g.\ compactness, definable sets, type-definable sets, type spaces, saturated models, indiscernible sequences) and descriptive set theory (e.g.\ Polish spaces, standard Borel spaces, Borel classes).

Furthermore, we will also use some well-known (but less widely known) facts and terms related to the following subjects.
\begin{enumerate}
\item
Borel cardinalities of Borel equivalence relations (\autocite[in particular Chapter 5]{Kan08}, \autocite[esp.\ Chapter 3]{BK96}). For a concise exposition of fundamental issues concerning this topic, the reader is referred to the preliminary sections in \autocite{KPS13} or \autocite{KMS14}. Let us only recall here that for Borel equivalence relations $E$ and $F$ on Polish (or, more generally, standard Borel) spaces $X$ and $Y$, respectively, we say that $E$ is Borel reducible to $F$, or that the Borel cardinality of $E$ is less than or equal to the Borel cardinality of $F$ (symbolically $E\leq_BF$) if there is a Borel reduction from $E$ to $F$, i.e.\ a Borel function $f : X \to Y$ such that $x_0 \Er x_1 \iff f(x_0) \Fr f(x_1)$ for all $x_0,x_1 \in X$; the relations $E$ and $F$ are Borel bireducible, or of the same Borel cardinality (symbolically $E \sim_B F$) if $E \leq_B F$ and $F \leq_B E$. The relation $E$ is smooth if $E \leq _B \Delta(Z)$, where $Z$ is a Polish space and $\Delta(Z)$ is the equality on $Z$.
\item
Strong Choquet topological spaces (\autocite{Kec95}).
\item
Lascar and Kim-Pillay strong types (\autocite[esp.\ first section]{CLPZ01}). In this paper, the relation of having the same Lascar strong type (i.e.\ the finest bounded invariant equivalence relation) will be denoted by $\equiv_L$, and the relation of having the same Kim-Pillay strong type (i.e.\ the finest bounded type-definable equivalence relation) will be denoted by $\equiv_{\KP}$.
\item
Model-theoretic connected group components (\autocite{GN08} and \autocite[first two sections]{Gis11}). Recall that for a group $G$ definable in the monster model, $G^{000}$ denotes the smallest invariant subgroup of bounded index, and $G^{00}$ -- the smallest type-definable subgroup of bounded index (both considered here without parameters, as explained in the introduction).
\item
The logic topology (\autocite[Section 2]{Pil04}). Recall that when $E$ is a type-definable equivalence relation on a type-definable set $X$, then a subset $D \subseteq X/E$ is closed if and only if its preimage by the quotient map is type-definable with parameters.
\end{enumerate}

\section{Framework}

\subsection{Bounded invariant equivalence relations}
In this chapter, we extend the theory of Borel cardinality of Lascar strong types as considered in \autocite{KPS13} to general invariant and bounded equivalence relations, to provide a uniform way of viewing bounded, invariant equivalence relations as relations on topological spaces, which will be standard Borel spaces in the countable case.

\begin{dfn}
Suppose $P$ is a product of sorts of $\mon$. We say that $P$ is \emph{countable}  if it is a product of countably many sorts.
\end{dfn}

\begin{dfn}
Suppose $X$ is a subset of some product of sorts $P$. Then we say that $P$ is the \emph{support} of $X$, and we say that $X$ is \emph{countably supported} if $P$ is countable (according to the preceding definition), and, more generally, say that it is \emph{$\lambda$-supported} for a cardinal $\lambda$ if $P$ is a $\lambda$-fold product.
\end{dfn}

\begin{dfn}[Borel invariant set, Borel class of an invariant set]
For any invariant set $X$, we say that $X$ is \emph{Borel} if the corresponding subset of $S(\emptyset)$ is, and in this case by \emph{Borel class} of $X$ we mean the Borel class of the corresponding subset of $S(\emptyset)$ (e.g.\ we say that $X$ is $F_\sigma$ if the corresponding set in $S(\emptyset)$ is $F_\sigma$, and we might say that $X$ is clopen if the corresponding subset of $S(\emptyset)$ is clopen, i.e.\ if $X$ is definable).

Similarly if $X$ is $A$-invariant, we say that it is \emph{Borel over $A$} if the corresponding subset of $S(A)$ is (and Borel class is understood analogously).

We say that a set is \emph{pseudo-closed} if it is closed over some small set (equivalently, if it is type-definable with parameters from a small set).
\end{dfn}

\begin{rem}
Notice that if both the language and $A$ are countable and $X$ is countably supported and Borel over $A$, then $S_X(A)$ -- endowed with the $\sigma$-algebra generated by formulas over $A$ -- is a standard Borel space.
\end{rem}

We will use the following descriptive-set-theoretic lemma several times.
\begin{lem}[{\autocite[Exercise 24.20]{Kec95}}]
\label{lem:cons}
Suppose $X,Y$ are compact, Polish spaces and $f\colon X\to Y$ is a continuous, surjective map. Then $f$ has a Borel section, so in particular, for any $B\subseteq Y$, $f^{-1}[B]$ is Borel if and only if $B$ is. Moreover, if they are Borel, then the two are of the same Borel class.
\end{lem}

The next corollary says that, in the countable case, when $X$ is invariant over a countable model, we need not specify the parameter set in order to talk about the Borel class of $X$. It is a generalisation of a well-known fact for sets which are definable or type-definable with parameters.
\begin{cor}
\label{cor:bcomp}
Let $A,B$ be any small sets. Suppose $X$ is an $A$-invariant and $B$-invariant subset of a small product of sorts.  Then if the support of $X$, the language, $A$ and $B$ are all countable, then the Borel class of $X$ over $A$ is the same as the Borel class of $X$ over $B$ (in particular, $X$ is Borel over $A$ if and only if it is Borel over $B$).

Without assumptions of countability, if $X$ is closed or $F_\sigma$ over $A$, it is also closed or $F_\sigma$ (respectively) over $B$.
\end{cor}
\begin{proof}
Without loss of generality, we can assume that $A\subseteq B$. Then the restriction map $f\colon S(B)\to S(A)$ is a continuous surjection, and $f^{-1}[X_A]=X_B$, so by Lemma~\ref{lem:cons}, we get the result for the first part.

The second part is true because $S(B)\to S(A)$ is continuous, as well as closed (as a continuous map between compact spaces).
\end{proof}

The following definition is somewhat self-explanatory, but since we are going to use it quite often, it should be stated explicitly.
\begin{dfn}
We say that an invariant equivalence relation $E$ on $X$ \emph{refines type} if for any $a,b\in X$ whenever $a\Er b$, then $a\equiv b$ (i.e.\ $\tp(a/\emptyset)=\tp(b/\emptyset)$). Equivalently, $E$ refines type if $E\subseteq {\equiv}\restr_X$.

Similarly, we say that $E$ \emph{refines Kim-Pillay strong type} $\equiv_{\KP}$ if $E\subseteq {\equiv_{\KP}}\restr_X$ and likewise we say that \emph{Kim-Pillay type refines $E$} if ${\equiv_{\KP}}\restr_X\subseteq E$.
\end{dfn}

The next definition is very important; it will be used to interpret a bounded, invariant equivalence relation $E$ as an abstract equivalence relation on a Polish space. It is a mild generalisation of $E_L^M$ and $E_{\KP}^M$ as introduced in \autocite{KPS13}.
\begin{dfn}
Suppose $E$ is a bounded, invariant equivalence relation on an invariant set $X$, while $M$ is a model.

Then we define $E^M\subseteq S_X(M)^2\subseteq S(M)^2$ as the relation
\[
p \Er^M q \iff \textrm{there are some $a\models p$ and $b\models q$ such that }a\Er b.
\]
(And the next proposition tells us that $E$-classes are $M$-invariant, so this is equivalent to saying that for all $a\models p,\, b\models q$ we have $a\Er b$, which implies that $\Er^M$ is an equivalence relation.)
\end{dfn}
The next proposition shows that $E^M$ is well-behaved in the sense explained in parentheses, and the Borel classes of $E^M$ and $E$ are the same in the countable case (which justifies the definition of Borel class of $E$ at the beginning of this subsection).
\begin{prop}[generalisation of {\autocite[Remark 2.2(i)]{KPS13}}]
\label{prop:eqcmp}
Consider a model $M$, and some bounded, invariant equivalence relation $E$ on an invariant subset $X$ of a product of sorts $P$.

Consider the natural restriction map $\pi\colon S_{P^2}(M)\to S_P(M)^2$ (i.e.\ \\ $\pi(\tp(a,b/M))=(\tp(a/M),\tp(b/M))$). Then we have the following facts:
\begin{itemize}
\item Each $E$-class is $M$-invariant, in particular, for any $a,b\in X$
\[a \Er b\iff \tp(a,b/M)\in {\Er}_M \iff \tp(a/M)\Er^M \tp(b/M)\]
and $\pi^{-1}[E^M]=E_M$.
\item
If one of $E^M$, $E_M$, $E$ (considered as a subset of $S_{P^2}(\emptyset)$) is closed or $F_\sigma$, then all of them are closed or $F_\sigma$ (respectively). In the countable case (when the support of $E$, the language and $M$ are all countable), we have more generally that the Borel classes of $E^M, E_M, E$ are all the same.
\item
Similarly -- for $M$-invariant $Y\subseteq X$ -- the relation $E^M\restr_{Y_M}$ is closed or $F_\sigma$ [or Borel in the countable case] if and only if $E_M\cap (Y^2)_M$ is.
\end{itemize}

\end{prop}
\begin{proof}
For the first bullet, notice that $E$ is refined by (a restriction of) Lascar strong type (cf.\ \autocite[Fact 1.4]{CLPZ01}), which in turn is refined by equivalence over $M$ (for any model $M$, cf.\ \autocite[Fact 1.12]{CLPZ01}), and therefore any points equivalent over $M$ are also Lascar equivalent, and hence $E$-equivalent.

The second bullet is similar to Corollary~\ref{cor:bcomp}: it is a consequence of the fact that $\pi$ and the restriction map $S_{P^2}(M)\to S_{P^2}(\emptyset)$ are both continuous and closed (because $S_{P^2}(M)$ is compact). For the countable case, we use Lemma~\ref{lem:cons}.

The last part follows analogously, as $\pi^{-1}[E^M\restr_{Y_M}]=E_M\cap (Y^2)_M$.
\end{proof}

The next two facts will be used in conjunction with Corollary~\ref{cor:bcomp} to show that some $E$-saturated sets (where $E$ is a bounded, invariant equivalence relation) are closed or $F_\sigma$ over any model $M$.

\begin{cor}
\label{cor:satinv}
If $E$ is a bounded, invariant equivalence relation on $X$ and $Y\subseteq X$ is $E$-saturated (i.e.\ containing any $E$-class intersecting it), then for any model $M$, $Y$ is $M$-invariant.
\end{cor}
\begin{proof}
Since $Y$ is $E$-saturated, it is a union of $E$-classes, each of which is setwise $M$-invariant.
\end{proof}
We also have a variant for groups.
\begin{cor}
\label{cor:satinvg}
If $G$ is an invariant group and $H$ is a subgroup of $G$ containing some invariant subgroup of bounded index (equivalently, $H$ contains $G^{000}$), then every coset of $H$ (including $H$ itself) is invariant over any model $M$.
\end{cor}
\begin{proof}
Immediate from the previous corollary with $E$ being the relation of being in the same coset of $G^{000}$.
\end{proof}

The next proposition establishes a notion of Borel cardinality.
\begin{prop}[generalisation of {\autocite[Proposition 2.3]{KPS13}}]
\label{prop:cartdf}
Assume that the language is countable. Let $E$ be a bounded (invariant) Borel equivalence relation on some type-definable and countably supported set $X$, and suppose $Y\subseteq X$ is pseudo-closed and $E$-saturated. Then the Borel cardinality of the restriction of $E^M$ to $Y_M$ does not depend on the choice of the countable model $M$. In particular, for $X=Y$, the Borel cardinality of $E^M$ does not depend on the choice of the countable model $M$.
\end{prop}
\begin{proof}
Follows from Lemma~\ref{lem:cons} analogously to \autocite[Proposition 2.3]{KPS13}. (Note that because of Proposition~\ref{prop:eqcmp} and Corollary~\ref{cor:satinv} the relations $E^M$ and $E^M\restr_{Y_M}$ are well-defined Borel equivalence relations on Polish spaces.)
\end{proof}
We have thus justified the following definition.
\begin{dfn}
If $E$ is as in the previous proposition, then by \emph{Borel cardinality of $E$} we mean the Borel cardinality of $E^M$ for a countable model $M$. Likewise, we say that $E$ is smooth if $E^M$ is smooth for a countable model $M$.

Similarly, if $Y$ is pseudo-closed and $E$-saturated, the \emph{Borel cardinality of $E\restr_Y$} is the Borel cardinality of $E^M\restr_{Y_M}$ for a countable model $M$.
\end{dfn}

\begin{fct}
\label{fct:tdsmt}
A bounded, type-definable equivalence relation is smooth. Similarly, if the restriction of a bounded, invariant equivalence relation to a saturated, pseudo-closed set $Y$ is relatively type-definable, then the restriction is smooth.
\end{fct}
\begin{proof}
If $E$ is type-definable, then so is its domain, and the corresponding subset of $S(M)^2$ is closed (by Proposition~\ref{prop:eqcmp}), and in particular $G_\delta$, and therefore smooth (cf.\ for example \autocite[Theorem 3.4.3]{BK96}). The proof of the second part is analogous: the Borel cardinality of the restriction of $E$ to $Y$ is the Borel cardinality of $E^M\cap (Y_M)^2$, which is closed in $(Y_M)^2$, and thus smooth.
\end{proof}

\subsection{Normal forms}
In this subsection, we introduce some more specific kinds of invariant equivalence relations, which naturally arise in the context of the main result.

\begin{dfn}[Normal form]
\label{dfn:nfrm}
If $\Phi_n(x,y)$ is a sequence of (partial) types on a type-definable set $X$ such that $\Phi_0(x,y)=((x=y)\land x\in X)$ and which is increasing (i.e.\ for all $n$, $\Phi_n(x,y)\vdash \Phi_{n+1}(x,y)$), then we say that \emph{$\biglor_{n\in \N} \Phi_n(x,y)$ is a normal form} for an invariant equivalence relation $E$ on $X$ if we have for any $a,b\in X$ the equivalence $a \Er b\iff \mon \models \biglor_{n\in \N}\Phi_n(a,b)$, and if the binary function $d=d_\Phi\colon X^2\to \N\cup \{\infty\}$ defined as 
\[
d(a,b)= \min \{n\in \N \mid \mon\models \Phi_n(a,b)\}
\]
(where $\min \emptyset=\infty$) is an invariant metric with possibly infinite values -- that is, it satisfies the axioms of coincidence, symmetry and triangle inequality. In this case, we say that \emph{$d$ induces $E$ on $X$}.
\end{dfn}

\begin{ex}
The prototypical example of a normal form is $\biglor_n d_L(x,y)\leq n$, inducing $\equiv_L$, and $d_L$ is the associated metric (where $\equiv_L$ is the relation of having the same Lascar strong type and $d_L$ is the Lascar distance).
\end{ex}

\begin{rem}
The Lascar distance, by its very definition, has the nice property that it is ``geodesic'' in the sense that if two points $a,b$ are at distance $n$, then there is a sequence of points $a=a_0,a_1,\ldots,a_n=b$ such that each pair of successive points is at distance $1$. The metrics obtained from normal forms usually will not have this property (notice that existence of such a ``geodesic'' metric for $E$ is equivalent to $E$ being the transitive closure of a type-definable relation).
\end{rem}

\begin{ex}
If $\Phi_n(x,y)$ is an increasing sequence of type-definable equivalence relations, then $\biglor_n\Phi_n(x,y)$ is trivially a normal form. In particular, if $E=\Phi(x,y)$ is type-definable, then we can put (for all $n>0$) $\Phi_n(x,y)=\Phi(x,y)$, yielding a somewhat degenerate normal form for $E$.
\end{ex}

\begin{dfn}
If we have an invariant equivalence relation $E$ on a type-definable set $X$ with a normal form $\biglor_{n\in \N} \Phi_n(x,y)$, corresponding to a metric $d$, and $Y\subseteq X$ is some nonempty set, then the \emph{diameter of $Y$} is the supremum of $d$-distances between points in $Y$.
\end{dfn}

\begin{fct}
\label{fct:stdiam}
If $E$ is as above, and $X$ is (the set of realisations of) a single complete type, then all $E$-classes have the same diameter (because the $\Aut(\mon)$ acts transitively on $X$ in this case, and the diameter is invariant under automorphisms).
\end{fct}

The following proposition is the essential step in adapting the techniques of \autocite{KMS14} to prove Theorem~\ref{thm:mainA}.
\begin{prop}
\label{prop:nfrm}
Suppose $E$ is an $F_\sigma$ (over $\emptyset$), bounded equivalence relation on a type-definable set $X$. Then $E$ has a normal form $\biglor_n \Phi_n$ such that $\Phi_1(x,y)$ holds for any $x,y$ which are terms of an infinite indiscernible sequence. (This implies that for any $a,b$, if $d_L(a,b)\leq n$, then $\models \Phi_{n}(a,b)$, so that the induced metric satisfies $d\leq d_L$. It also shows that every $F_\sigma$ equivalence relation has a normal form.)
\end{prop}
\begin{proof}
As $E$ is bounded, the Lascar strong type restricted to $X$ is a refinement of $E$ (cf.\ \autocite[Fact 1.4]{CLPZ01}), and hence $E\cup ({\equiv}_L\restr_X)=E$. In addition, since $E$ is $F_\sigma$, we can find types $\Phi_n(x,y)$ such that $x\Er y\iff \mon\models \biglor_n\Phi_n(x,y)$.

Consider the sequence $\Phi'_n(x,y)$ of types, defined recursively by:
\begin{enumerate}
\item
$\Phi_0'(x,y)=((x=y)\land x\in X)$,
\item
$\Phi_1'(x,y)=(\Phi_1(x,y)\lor \Phi_1(y,x)\lor x=y\lor d_L(x,y)\leq 1)\land (x,y\in X)$,
\item
$\Phi_{n+1}'(x,y)=\Phi_{n+1}(x,y)\lor \Phi_{n+1}(y,x)\lor (\exists z)(\Phi'_{n}(x,z)\land\Phi'_n(z,y))$.
\end{enumerate}
It is easy to see that $\biglor \Phi_{n}'$ is a normal form and represents the smallest equivalence relation containing $E$ and ${\equiv}_L\restr_X$ (as a set of pairs), which is just $E$, and $d_L(x,y)\leq 1$ (i.e.\ the statement that $x,y$ are in an infinite indiscernible sequence) implies $\Phi'_1(x,y)$ by the definition.

The statement in the parentheses follows from the fact that $d_L(a,b)\leq n$ is defined as the $n$-fold composition of $d_L(a,b)\leq 1$.
\end{proof}

The theorem of Newelski we will see shortly is a motivating example for the study of Borel cardinality: it can be interpreted as saying that some equivalence relations have Borel cardinality of at least $\Delta(2^\N)$. We will see later in Corollary~\ref{cor:maincont2} that for $E$ which are orbital (a concept which we will define soon), we can strengthen this result to replace $\Delta(2^\N)$ with $\EZ$, and this is optimal in the sense explained in a remark after Corollary~\ref{cor:maincont2}.
\begin{thm}[{[Corollary 1.12]\autocite{New03}}]
\label{thm:twN}
Assume $x \Er y$ is an equivalence relation refining $\equiv$, with normal form $\biglor_{n\in\N}\Phi_n$. Assume $p\in S(\emptyset)$ and $Y\subseteq p(\mon)$ is pseudo-closed and $E$-saturated. Then either $E$ is equivalent on $Y$ to some $\Phi_n(x,y)$ (and therefore $E$ is relatively type-definable on $Y$), or $\lvert Y/E\rvert\geq 2^{\aleph_0}$.
\end{thm}

\begin{rem}
Newelski uses a slightly more stringent definition of a normal form (which we may enforce in all interesting cases without any significant loss of generality), i.e.\ that $d$ satisfies not only triangle inequality, but also 
\[
d(a,b),d(b,c)\leq n\implies d(a,c)\leq n+1.
\]
The definition used in this paper is sufficient for the previous theorem, and in addition, it has the added benefit of being satisfied by the Lascar distance $d_L$, and it seems more natural in general.
\end{rem}

The following corollary allows us some freedom with regards to the normal form, allowing us to replace -- in some cases -- any normal form with one chosen as in Proposition~\ref{prop:nfrm}, without loss of generality.
\begin{cor}\label{cor:diam}
Suppose $E$ is an $F_\sigma$ equivalence relation on a type-definable set, and that $E$ refines $\equiv$. Then for any class $C$ of $E$, the following are equivalent:
\begin{enumerate}
\item $C$ is pseudo-closed,
\item $C$ has finite diameter with respect to each normal form of $E$ (i.e.\ it has finite diameter with respect to the metric induced by each normal form),
\item $C$ has finite diameter with respect to some normal form of $E$.
\end{enumerate}
In addition, if $E$ is bounded and all $E$-classes satisfy these conditions, then $E$ is refined by $\equiv_{\KP}$ (restricted to its domain).
\end{cor}
\begin{proof}
Assume that $C$ is pseudo-closed. Setting $Y=C$ in Theorem~\ref{thm:twN}, we immediately get that $C$ has finite diameter with respect to any normal form of $E$. Implication from the second condition to third follows from the fact that $E$ has a normal form by the previous proposition, and the implication from third to first is trivial.

``In addition'' can be obtained as follows. $E$ refines $\equiv$, so it is enough to show that the restriction of $E$ to any $p\in S(\emptyset)$ is refined by the restriction of $\equiv_{\KP}$ to $p$. But any class in the restriction has finite diameter with respect to some normal form, and they all have the same diameter (by Fact~\ref{fct:stdiam}), so in fact, the restriction is type-definable and as such refined by $\equiv_{\KP}$ (cf.\ \autocite[Fact 1.4]{CLPZ01}).
\end{proof}

\begin{ex}
\label{ex:nonf}
The above is no longer true if we allow $E$ to be refined by $\equiv$. For example, consider the theory $T=\Th (\R,+,\cdot,0,1,<)$ of real closed fields, and the total relation on the entire model. Clearly, it has a normal form $\{x=y\}\lor\biglor_{n>0}(x=x)$, and the induced metric is just the discrete $0$-$1$ metric, and in particular its only class (the entire model) has diameter $1$. On the other hand, we might give it a normal form $\{x=y\}\lor\biglor_{n>0}(\bigland_{m\geq n} (x= m\liff y= m))$ (where $m$ ranges over natural numbers). With respect to this normal form, any two distinct positive natural numbers $k,l$ are at distance $\max(k,l)+1$. In particular, the diameter of the only class is infinite.
\end{ex}

\begin{rem}
If $E$ is a type-definable equivalence relation, then its classes are trivially pseudo-closed, so by Corollary~\ref{cor:diam}, if $E$ refines $\equiv$, then for any normal form of $E$, all $E$-classes have finite diameter.
\end{rem}

\subsection{Orbital equivalence relations}

For technical reasons, later on we will rely on the action of a group of automorphisms, so we introduce the following definition.
\begin{dfn}[Orbital equivalence relation, orbital on types equivalence relation]
Suppose $E$ is an invariant equivalence relation on a set $X$.
\begin{itemize}
\item
We say that $E$ is \emph{orbital} if there is a group $\Gamma\leq \Aut(\mon)$ such that $\Gamma$ preserves classes of $E$ setwise and acts transitively on each class.
\item
We say that $E$ is \emph{orbital on types} if it refines type and the restriction of $E$ to any complete $\emptyset$-type is orbital.
\end{itemize}
\end{dfn}

\begin{rems}$\,$
\begin{itemize}
\item
The fact that a given relation is orbital is witnessed by one group $\Gamma$ (which is not necessarily unique), whereas the fact that it is orbital on types is witnessed by a collection of groups (one group for each complete $\emptyset$-type).
\item
An orbital equivalence relation always refines type. (So every orbital equivalence relation is orbital on types.)
\item
The relations $\equiv_L,\equiv_{\KP}$ are orbital (as witnessed by $\Autf{L}(\mon),\Autf{KP}(\mon)$).
\item
The group witnessing that a given relation is orbital can always be chosen as a normal subgroup of $\Aut(\mon)$ (as we can replace it with its normal closure).
\end{itemize}
\end{rems}

The following proposition shows that the definition of an orbital on types equivalence relation is, in a way, the weakest possible for the proof of Theorem~\ref{thm:mainA}.

\begin{prop}
An invariant equivalence relation $E$ refining type is orbital on types if and only if for any class $C$ of $E$ there is a group of automorphisms $\Gamma$ which preserves $E$ classes within the (complete $\emptyset$-)type $p$ containing $C$, and acts transitively on $C$.
\end{prop}
\begin{proof}
The implication $(\Rightarrow)$ is clearly a weakening. For $(\Leftarrow)$, observe that $\Aut(\mon)$ acts transitively on $X:=p(\mon)$, so for any class $C'\in X/E$ we have an automorphism $\sigma$ which takes $C$ to $C'$. It is easy to see that then $\sigma \Gamma \sigma^{-1}$ acts transitively on $C'$ and preserves all $E$-classes in $X$ setwise. From that we conclude that the normal closure of $\Gamma$ in $\Aut(\mon)$ witnesses that $E$ restricted to $X$ is orbital.
\end{proof}

The following simple corollary allows us to easily recognise some relations as orbital on types.
\begin{cor}
\label{cor:2cl}
If $E$ is an invariant equivalence relation on an invariant set $X$, refining $\equiv$, and the restriction of $E$ to any complete type in $X$ has at most two classes, then $E$ is orbital on types.
\end{cor}
\begin{proof}
Without loss of generality we may assume that $X$ is a single complete type, so $\Aut(\mon)$ acts transitively on $X$. In particular, for any element $a\in X$, we have a set $S\subseteq \Aut(\mon)$ such that $S\cdot a=[a]_E$. Since $E$ is invariant, elements of $S$ preserve $[a]_E$ and so does the group $\Gamma=\langle S\rangle$.

Of course, $\Gamma$ preserves $X$, so it also preserves the complement $X\setminus [a]_E$. But since $E$ has at most two classes, this means that $\Gamma$ preserves all classes, so by the previous proposition, $E$ is orbital on types.
\end{proof}

At a glance, it is not obvious whether the condition that $E$ is orbital on types is any stronger than the condition that it refines type. The following examples show that it is indeed the case.

\begin{ex}
\label{ex:finperm}
Consider the permutation group
\begin{align*}
G&=\langle (1,2)(3,5)(4,6), (1,3,6)(2,4,5)\rangle\\
&=\{(), (1,2)(3,5)(4,6), (1,3,6)(2,4,5),\\
&(1,4)(2,3)(5,6), (1,5)(2,6)(3,4), (1,6,3)(2,5,4)\}
\end{align*} acting naturally on a $6$-element set. Then the equivalence relation $\sim$ such that $1\sim 2,\,3\sim 4,\,5\sim 6$ (and no other nontrivial relations) is preserved by $G$, but it is not the orbital equivalence relation of any subgroup (in fact, the only element of $G$ which preserves all $\sim$-classes setwise is the identity).

Let $M_0$ be a structure with base set $\{1,2,3,4,5,6\}$, with a relation symbol $E$ for $\sim$, and such that $G$ is the automorphism group of $M_0$ (which we can obtain, for instance, by adding a predicate for the set of all orbits of $G$ on $M_0^6$).

Then $E$ is an invariant (even definable) equivalence relation which refines $\equiv$ and is not orbital on types.
\end{ex}
We can extend Example~\ref{ex:finperm} to an infinite model in a number of simple ways, for instance, by taking a product with an infinite trivial structure.

We finish with a less artificial example.
\begin{ex}
\label{ex:trace}
Consider a large algebraically closed field $K$ of characteristic $p>0$, and choose some $t\in K$, transcendental over the prime field ${\bf F}_p$, and consider $T=\Th(K,+,\cdot,t)$.

Let $n>3$ be a natural number which is not divisible $p$, and $X$ be the set of $n$-th roots of $t$ in $K$ (i.e.\ the roots of $x^n-t$). Notice that $X$ generates a definable, finite additive group $\langle X\rangle$. Let us introduce
\[
G=(\{a=(a_1,a_2)\in K^2 \mid a_1+a_2\in \langle X\rangle\},+).
\]
$G$ is a definable group (definably isomorphic to $K\times \langle X\rangle$). Consider the equivalence relation on $G$ defined by
\[
a \Er b \iff (a\equiv b\land a_1+a_2=b_1+b_2).
\]
We will show that $E$ is not orbital on types, even though it is type-definable, bounded and refines $\equiv$. (N.b.\ this $E$ is the conjunction of $\equiv$ and the relation of lying in the same coset of $G^{000}$, which in this case is equal to $G^{0}$.)

Let $\xi$ be some primitive $n$th root of unity. 
One can easily check that  for any $x_1,x_2\in X$, the pairs $(x_1,\xi)$ and $(x_2,\xi^{-1})$ have the same type, which implies that all $a\in G$ of the form $(x,\xi^{\pm 1} x)$, where $x\in X$, have the same type, say $p_0\in S_G(\emptyset)$.
For any $x\in X$ we also have $(x,\xi x) \Er (\xi x,x)$. Thus, if $E$ was orbital on types, there would be some automorphism $f\in \Aut(K/t)$ which takes $x$ to $\xi x$ and $\xi x$ to $x$ -- therefore taking $\xi$ to $\xi^{-1}$ -- which preserves setwise the $E$-classes within $p_0$. 
But then
\[
b=f((\xi x,\xi^2 x))=(x,\xi^{-1} x) \mathrel{\neg E} (\xi x,\xi^2 x)=a\models p_0,	
\]
because $a_1+a_2-b_1-b_2=x(\xi+\xi^2-1-\xi^{-1})=\xi^{-1} x(\xi^3+\xi^2-\xi^1-1)$ and $\xi$ is algebraic of degree $n>3$.

We have seen that the $E$-class of $(\xi x,\xi^2 x)\models p_0$ is not preserved by $f$, a contradiction.
\end{ex}

\subsection{Invariant subgroups as invariant equivalence relations}
We start from the following natural definition.

\begin{dfn}
Suppose $G$ is a type-definable group and $H\leq G$ is invariant. We define $E_H$ as the relation on $G$ of lying in the same right coset of $H$.
\end{dfn}

\begin{rem}
Clearly, $E_H$ is invariant, and it has $[G:H]$ classes, so $H$ has bounded index if and only if $E_H$ is a bounded equivalence relation.
\end{rem}

It is not hard to see that invariant subgroups of type-definable groups correspond to invariant equivalence relations as shown in the following lemma.
\begin{lem}
\label{lem:compgrp}
Suppose $G$ is a type-definable group and $H\leq G$ is an invariant subgroup. Then $E_H$ is type-definable or $F_\sigma$ if and only if $H$ is type-definable or $F_\sigma$, respectively.
\end{lem}
\begin{proof}
Consider the mapping $f\colon S_{G^2}(\emptyset)\to S_G(\emptyset)$ given by $\tp(a,b/\emptyset)\mapsto \tp(ab^{-1}/\emptyset)$.
Since the operations in $G$ are type-definable, this map is a well-defined, continuous and closed (by compactness) surjection, and $E_H=f^{-1}[H]$.
\end{proof}

\begin{rems} $\,$
\begin{itemize}
\item
The previous lemma would remain true if we had taken for $E_H$ the relation of lying in the same left coset, but right cosets will be technically more convenient in a short while.
\item
Equivalence relations $E_H$ do not refine type, and in particular are not orbital on types, which will be needed later on. We will resolve this issue shortly by choosing a different equivalence relation to represent $H$, which will be closely related to $E_H$ (in a way, homeomorphically equivalent) and orbital on types for normal $H$.
\end{itemize}
\end{rems}

The theorem below will allow us to ``transform'' the relation $E_H$ to an equivalence relation on a single type.

\begin{thm}[see {\autocite[Section 3, in particular Propositions 3.3 and 3.4]{GN08}}]
\label{thm:phd}$\,$\\
If $G$ is a definable group, and we adjoin to $\mon$ a left principal homogeneous space $\mathfrak X$ of $G$ (as a new sort; we might think of it as an ``affine copy of $G$''), along with a binary function symbol for the left action of $G$ on $\mathfrak{X}$, then the Kim-Pillay and Lascar strong types correspond exactly to the orbit equivalence relations of $G^{00}$ and $G^{000}$ acting on $\mathfrak X$. Moreover, we have isomorphisms:
\begin{align*}
\Aut((\mon,{\mathfrak X},\cdot))&\cong G\rtimes \Aut(\mon), \\
\Autf{KP}((\mon,{\mathfrak X},\cdot))&\cong G^{00}\rtimes \Autf{KP}(\mon), \\
\Autf{L}((\mon,{\mathfrak X},\cdot))&\cong G^{000}\rtimes \Autf{L}(\mon).
\end{align*}
Where:
\begin{enumerate}
\item
the semidirect product is induced by the natural action of $\Aut(\mon)$ on $G$,
\item
on $\mon$, the action of $\Aut(\mon)$ is natural, and that of $G$ is trivial,
\item
on $\mathfrak{X}$ we define the action by fixing some $x_0$ and putting $\sigma_g(h\cdot x_0)=(hg^{-1})x_0$ and $\sigma(h\cdot x_0)=\sigma(h)\cdot x_0$ (for $g\in G$ and $\sigma\in \Aut(\mon)$).
\end{enumerate}
\end{thm}

\begin{rem}
The isomorphisms are not canonical in general: they depend on the choice of the base point $x_0$.
\end{rem}

Until the end of this subsection, we fix a definable group $G$ and the structure $(\mon,\mathfrak{X},\cdot)$ as above. Note that a definable group is always finitely (and therefore countably) supported.

\begin{dfn}
Let $H$ be an invariant subgroup of $G$. Then $E_{H,X}$ is the relation on $\mathfrak X$ of being in the same $H$-orbit.
\end{dfn}

\begin{prop}
\label{prop:grrel}
The mapping $\Phi\colon H\mapsto E_{H,X}$ is a bijection between invariant subgroups of $G$ and invariant equivalence relations on $\mathfrak{X}$.
\end{prop}
\begin{proof}
We fix some $x_0\in \mathfrak X$, so as to apply the description of the automorphism group of $(\mon,{\mathfrak X},\cdot)$ from Theorem~\ref{thm:phd}.

First, choose  some invariant $H\leq G$. We will show that $E_{H,X}$ is invariant.
By the definition of $E_{H,X}$ and Theorem~\ref{thm:phd}, it is enough to show that for arbitrary $h\in H$, $\sigma \in \Aut(\mon)$ and $g,k\in G$, one has  $\sigma(kx_0)\Er_{H,X} \sigma(hkx_0)$ and $kgx_0\Er_{H,X} hkgx_0$. The latter is immediate by the definition of $E_{H,X}$. For the former, just see that
\[
\sigma(kx_0)=\sigma(k)x_0\Er_{H,X} \sigma(h)\sigma(k)x_0=\sigma(hkx_0),
\]
because $\sigma(h)\in H$ (by invariance of $H$).

To see that $\Phi$ is a bijection, choose an arbitrary invariant equivalence relation $E$ on $\mathfrak X$, and let $H$ be the setwise stabiliser of $[x_0]_E$. Take arbitrary $h\in H$, $\sigma\in \Aut(\mon)$. Then
\[
x_0 \Er hx_0 \implies x_0=\sigma(x_0) \Er \sigma(hx_0)=\sigma(h)x_0,
\]
therefore $\sigma(h)\in H$, and since $h$ and $\sigma$ were arbitrary, $H$ is invariant. To see that $E=E_{H,X}$, notice that for any $x_1=k_1x_0$ and $x_2=k_2x_0$ we have
\[
k_1x_0 \Er k_2x_0\iff x_0 \Er k_2k_1^{-1}x_0\iff k_2k_1^{-1}\in H\iff (\exists h\in H)\, hk_1x_0=k_2x_0. \qedhere
\]
\end{proof}

\begin{rem}
An invariant subgroup $H\leq G$ has bounded index if and only if $E_{H,X}$ is a bounded equivalence relation.
\end{rem}

\begin{prop}
\label{prop:grres}
Let $H\leq G$ be an invariant subgroup of bounded index and let $K$ be a pseudo-closed subgroup such that $H\leq K\leq G$.

Let $M\preceq \mon$ be any small model. Then, if we put $N=(M,G(M)\cdot x_0)\preceq (\mon,\mathfrak X,\cdot)$, the map $g\mapsto g\cdot x_0$ induces a homeomorphism $S_G(M)\to S_X(N)$ which takes $E_{H}^M$ to $E_{H,X}^N$ and $K_M$ to $(K\cdot x_0)_N$.

In particular:
\begin{itemize}
\item
$E_{H,X}$ is closed or $F_\sigma$ if and only if $E_H$ is (respectively), 
\item
if the language and $M$ are both countable, while $H$ is $F_\sigma$ (or even Borel), then the Borel cardinalities of $E_H\restr_{K}$ and $E_{H,X}\restr_{K\cdot x_0}$ coincide.
\end{itemize}
\end{prop}
\begin{proof}
The map $f\colon S_G(N)\to S_X(N)$ defined by $f(\tp(g/N))=\tp(g\cdot x_0/N)$ is a homeomorphism (because it is induced by an $N$-definable bijection), and $f$ takes $K_N$ to $(K\cdot x_0)_N$ and $E_H^N$ to $E_{H,X}^N$.
It is also easy to see that the restriction map $g\colon S_G(N)\to S_G(M)$ (with the latter considered in the original structure $\mon$) is also a homeomorphism, which takes $E^N_H$ to $E^M_H$ and $K_N$ to $K_M$. The rest is now clear.
\end{proof}

Using this language, we have the following corollary of Theorem~\ref{thm:twN}:
\begin{cor}
\label{cor:sumt}
Suppose $G$ is a definable group and $H\leq G$ is a type-definable subgroup. Suppose in addition that $H=\bigcup C_n$, where $C_n$ are type-definable, symmetric sets containing $e$ and such that $C_n^2\subseteq C_{n+1}$. Then for some $n$ we have $H=C_n$.
\end{cor}
\begin{proof}
Consider the equivalence relation $E_{H,X}$. Then $C_{n,X}:=\{(x,x')\in \mathfrak X^2\mid x\in C_n x'\}$ give us a normal form for this relation, which is type-definable and only defined on a single type, so the result follows from Corollary~\ref{cor:diam}.
\end{proof}

We finish with an observation that allows us to easily see that some $E_{H,X}$ are orbital.

\begin{prop}
\label{prop:grnrm}
Suppose $H$ is a normal, invariant subgroup of $G$. Then $E_{H,X}$ is orbital as witnessed by $H\leq\Aut((\mon,\mathfrak X,\cdot))$.
\end{prop}
\begin{proof}
Consider the action $*$ of $H$ on $(\mon,\mathfrak X,\cdot)$ by automorphisms. Then – because $H$ is a normal subgroup of $G$ – we have for any $x=g\cdot x_0\in \mathfrak X$ that
\[
H * (g\cdot x_0)= (gH^{-1})\cdot x_0= (gH)\cdot x_0 =(Hg)\cdot x_0=H\cdot (g\cdot x_0)=[x]_{E_{H,X}},
\]
and hence $H\leq G\rtimes \Aut(\mon)=\Aut((\mon,\mathfrak X,\cdot))$ witnesses that $E_{H,X}$ is orbital.
\end{proof}

\begin{rem}
The converse of the previous proposition is not true: if we have $G=S_3$, $H=\langle (1,2)\rangle$ and $\Aut(\mon)$ acting on $G$ in such a way that any $\sigma\in \Aut(\mon)$ acts on $G$ either trivially or by conjugation by $(1,2)$, then although $H$ is not normal, $E_{H,X}$ is orbital: for $\sigma\in \Aut(\mon)$ acting nontrivially on $G$ we have
\[
((1,2)^{-1},\sigma)(g\cdot x_0)=((1,2)\cdot g\cdot (1,2)^{-1})\cdot (1,2)\cdot x_0=(1,2)\cdot (g\cdot x_0).
\]
\end{rem}

\section{The technical theorem}
\subsection{The countable case}
As before, when $E$ is an invariant, bounded equivalence relation, we denote by $E^M$ the induced equivalence relation on $S(M)$. For the statement of the next corollary, we need to extend the notion of distance to the type spaces.

\begin{dfn}
If $E$ is an $F_\sigma$ equivalence relation induced by a metric $d$ (coming from some normal form), then we also denote by $d_M$ the induced distance on $S(M)$, i.e.\
\[
d_M(p_1,p_2)=\min_{a_1\models p_1,a_2\models p_2} d(a_1,a_2).
\]
\end{dfn}

\begin{rem}
The classes of $E^M$ are precisely the ``metric components'' of $d_M$, i.e.\ the maximal sets of types which are pairwise at finite distance from one another in the sense of $d_M$, though $d_M$ might not satisfy the triangle inequality, so it is not in general a metric.
\end{rem}

We will use the next theorem to show Theorem~\ref{thm:mainA}.
\begin{thm}[based on {\autocite[Corollary 2.3]{KMS14}}] \label{thm:tool}
Suppose we have:
\begin{itemize}
\item
a countable theory $T$ with monster model $\mon$,
\item
a countable model $M\preceq \mon$,
\item
a type-definable, countably supported set $X$,
\item
a bounded $F_\sigma$ equivalence relation $E$ on $X$, with normal form $\biglor_n\Phi_n$, inducing metric $d$,
\item
a pseudo-closed and $E$-saturated $Y\subseteq X$.
\end{itemize}
Assume in addition that there is some $p\in Y_M\subseteq S_X(M)$ such that for every formula $\varphi\in p$ with parameters in $M$, and for all $N\in \N$, there is some $\sigma\in \Aut(\mon)$ such that:
\begin{enumerate}
\item
$\sigma$ fixes $M$ and all $E$-classes in $Y$ setwise (and therefore $Y$ itself as well),
\item
$\varphi\in \sigma(p)$ and $N<d_M(\sigma(p),p)$.
\end{enumerate}
Then there is a continuous, injective homomorphism
\[
(2^\N,{\EZ},\neg {\EZ})\to(Y_M,E^M\restr_{Y_M},\neg (E^M\restr_{Y_M})).
\]
In particular, $E^M\restr_{Y_M}$ is not smooth.
\end{thm}
\begin{proof}
The proof is the same as that of \autocite[Corollary 2.3]{KMS14}. The only difference is that for $\Gamma$ we take the group of automorphisms of $\mon$ which fix $M$ and all $E$-classes in $Y$ setwise (instead of all Lascar strong types as there), and we use $d_M$ instead of the Lascar distance. Note that $Y_M$ is Polish by Corollary~\ref{cor:bcomp} and Corollary~\ref{cor:satinv}.
\end{proof}

The above implies the next theorem. As mentioned in the introduction, a similar theorem has been proved, independently, in \autocite{KM14} using different methods. The proof we give here is a generalization of the main result of \autocite{KMS14}, where the relation in question is the Lascar strong type.

\begin{thm}[based on {\autocite[Theorem 4.13]{KMS14}}] \label{thm:mainA}
We are working in the monster model $\mon$ of a complete, countable theory. Suppose we have:
\begin{itemize}
\item
a type-definable, countably supported set $X$,
\item
a bounded, $F_\sigma$ equivalence relation $E$ on $X$, which is orbital on types,
\item
a pseudo-closed and $E$-saturated set $Y\subseteq X$,
\item
an $E$-class $C\subseteq Y$ with infinite diameter with respect to some normal form of $E$,
\end{itemize}
Then $E\restr_Y$ is not smooth.
\end{thm}

\begin{proof}
By Proposition~\ref{prop:nfrm} and Corollary~\ref{cor:diam}, we can choose  a normal form for $E$ such that the induced distance $d$ satisfies $d\leq d_L$, with respect to which $C$ has infinite diameter.
We can also assume that $X$ is the complete type containing $C$ (by restricting $Y$ to this type), so that $E$ is orbital as witnessed by some group $\Gamma$.

Then we proceed as in Theorem 4.13 of \autocite{KMS14} (aiming to use Theorem~\ref{thm:tool}), only instead of $\Autf L(\mon)$ we use $\Gamma$ (note that all the facts about generic and proper types and formulas from \autocite{KMS14} still hold with $\Gamma$ replacing $\Autf L(\mon)$, because $\Gamma$ acts transitively on $C$), and instead of Lascar distance we use $d$.
\end{proof}

\begin{rem}
We can always take for $\Gamma$ the group of all automorphisms preserving $E$-classes setwise. (In which case $\Gamma\unlhd\Aut(\mon)$.)
\end{rem}

The next corollary can be seen as a strengthening of Theorem~\ref{thm:twN} in case of $E$ which are orbital on types (because a relation with countably many classes is smooth).
\begin{cor}\label{cor:maincont2}
Assume that the language is countable. Suppose $E$ is a bounded, $F_\sigma$ and orbital on types equivalence relation on a type-definable and countably supported set $X$. Let $a\in X$ be arbitrary, and assume that $Y\subseteq [a]_\equiv$ is $E$-saturated, pseudo-closed with $a\in Y$. Fix any normal form $\biglor_n\Phi_n$ for $E$. Then the following are equivalent:
\begin{enumerate}
\item
$E\restr_Y$ is smooth,
\item
$E\restr_{[a]_\equiv}$ is type-definable,
\item
all $E$-classes in $[a]_{\equiv}$ have finite diameter with respect to $\biglor_n\Phi_n$,
\item
all $E$-classes in $[a]_{\equiv}$ are pseudo-closed,
\item
$[a]_E$ has finite diameter with respect to $\biglor_n\Phi_n$,
\item
\label{it:apcl}
$[a]_E$ is pseudo-closed.
\end{enumerate}
\end{cor}
\begin{proof}
We may assume without loss of generality that $X=[a]_{\equiv}$. Then $E$ is orbital.

All the conditions imply that $[a]_E$ is pseudo-closed (the first one does by Theorem~\ref{thm:mainA}, and the others are clearly stronger than \eqref{it:apcl}).

On the other hand, this condition implies that $[a]_E$ has finite diameter (by Theorem~\ref{thm:twN}), so all classes have the same, finite diameter (by Fact~\ref{fct:stdiam}), so of course they are pseudo-closed and $E$ is type-definable, and therefore $E\restr_Y$ is smooth (by Fact~\ref{fct:tdsmt}).
\end{proof}

\begin{rem}
Corollary~\ref{cor:maincont2} is, in a way, a strongest possible result. This is to say, there are examples of bounded, $F_\sigma$ and orbital equivalence relations whose Borel cardinality is exactly that of $\EZ$ (cf.\ \autocite[Example 3.3]{KPS13}), so we cannot replace the condition that $E\restr_Y$ is smooth with some weaker upper bound on Borel cardinality.
\end{rem}

For relations refining $\equiv_{\KP}$, we may be even more specific.
\begin{cor}\label{cor:refkp}
Assume that the language is countable.
Suppose $E$ is bounded, $F_\sigma$, countably supported and orbital on types. Suppose in addition that it refines $\equiv_{\KP}$. Then for any $a$ in the domain of $E$, we have that $E\restr_{[a]_{\equiv_{\KP}}}$ is trivial (i.e.\ total on $[a]_{\equiv_{\KP}}$) if and only if it is smooth. (In particular, if $E$ is smooth, then it is equal to a restriction of $\equiv_{\KP}$.)
\end{cor}
\begin{proof}
The implication from left to right is trivial. To prove the converse, choose any $a$ in domain of $E$.
The set $[a]_{\equiv_{\KP}}$ is $E$-saturated (because $E$ refines $\equiv_{\KP}$), type-definable over $a$ and contained in $[a]_{\equiv}$, so we can assume without loss of generality that $E$ is defined on $[a]_{\equiv}$. Then we can apply Corollary~\ref{cor:maincont2}, which tells us that if $E\restr_{[a]_{\equiv_{\KP}}}$ is smooth, then $E$ is type-definable. But in this case $E$ is refined by $\equiv_{\KP}$ (by \autocite[Fact 1.4]{CLPZ01}), and therefore equal to $\equiv_{\KP}$ restricted to $[a]_{\equiv}$, and so $E\restr_{[a]_{\equiv_{\KP}}}$ is trivial.
\end{proof}

We infer an analogous result for invariant subgroups of bounded index of definable groups, whose uncountable counterpart (Corollary~\ref{cor:mainGu}) will be employed in the final section in the context of definable group extensions.

\begin{cor}\label{cor:mainG}
Assume the language is countable.
Suppose that $G$ is a definable group (and therefore countably, and even finitely supported) and $H\unlhd G$ is an invariant, normal subgroup of bounded index, which is $F_\sigma$ (equivalently, generated by a countable family of type-definable sets). Suppose in addition that $K\geq H$ is a pseudo-closed subgroup of $G$. Then $E_H\restr_{K}$ is smooth if and only if $H$ is type-definable.
\end{cor}
\begin{proof}
If $H$ is type-definable, then by Lemma~\ref{lem:compgrp}, $E_H$ is a type-definable equivalence relation (on a type-definable set), and as such it is immediately smooth by Fact~\ref{fct:tdsmt}, and so is its restriction to $K$.

The proof in the other direction will proceed by contraposition: assume that $H$ is not type-definable. Recall Proposition~\ref{prop:grrel}: consider, once again, the sorted structure $(\mon,{\mathfrak X},\cdot)$.

By Proposition~\ref{prop:grres}, $H$ corresponds to a bounded $F_\sigma$ equivalence relation $E_{H,X}$ on $\mathfrak X$ (which is not type-definable, since $H$ is not), which is only defined on a single type, and -- owing to the assumption that $H$ is normal and Proposition~\ref{prop:grnrm} -- orbital.
Evidently $K\cdot x_0$ is $E_{H,X}$-saturated and pseudo-closed, so we can apply Corollary~\ref{cor:maincont2} to $E=E_{H,X}$ and $Y=K\cdot x_0$, deducing that $E_{H,X}\restr_{K\cdot x_0}$ is not smooth, and therefore (by Proposition~\ref{prop:grres}) neither is $E_H\restr_K$.
\end{proof}

\subsection{The uncountable case}
We intend to formulate the uncountable analogues of Theorem~\ref{thm:tool} and Theorem~\ref{thm:mainA}, but first we need to introduce some terminology.

\begin{dfn}
Suppose $L'\subseteq L$ is some sublanguage, $x'$ is a tuple of variables and $A$ is a set. Then by $L'_{x'}(A)$  we denote the Lindenbaum-Tarski algebra of (equivalence classes of) $L'$-formulas with free variables among $x'$ and parameters from $A$.
\end{dfn}

\begin{dfn}
Suppose we have an $F_\sigma$ equivalence relation $E$ with a normal form $\biglor_n\Phi_n(x,y)$. Suppose in addition that $L'\subseteq L$ is some sublanguage, $x'y'\subseteq xy$ is some smaller tuple of variables. Then we define the \emph{restriction of the normal form}, $\Phi_n\restr_{L'_{x'y'}(\emptyset)}$ as the set of $L'_{x'y'}(\emptyset)$-consequences of $\Phi_n(x,y)$, 
i.e.\ 
\[
\Phi_{n}\restr_{L'_{x'y'}(\emptyset)}=\{\varphi(x',y')\in L'_{x'y'}(\emptyset) \mid \Phi_n(x,y)\proves\varphi(x',y')\},
\]
and we define $E\restr_{L'_{x'y'}(\emptyset)}$ as the $F_\sigma$ relation given by
\[
a \mathrel{E\restr_{L'_{x'y'}(\emptyset)}} b\iff \mon\models \biglor_n\Phi_n\restr_{L'_{x'y'}(\emptyset)}(a,b).
\]
\end{dfn}

\begin{rem}
For arbitrary $L',x'y'$ and $E$, $\biglor_n \Phi_n\restr_{L'_{x'y'}(\emptyset)}$ might not be a normal form 
(it need not satisfy the triangle inequality, but see the next proposition), but if it is, $E\restr_{L'_{x'y'}(\emptyset)}$ is an equivalence relation coarser than $E$ (and with a larger domain) and the metric $d'$ associated with the restricted normal form satisfies $d'\leq d$.
\end{rem}

\begin{prop}
\label{prop:nfrst}
Given $L',x'y'$, we may always extend $L'$ and $x'y'$ (without increasing their cardinality by more than $\lvert L'\rvert+\lvert x'y'\rvert+\aleph_0$) to $L'',x''y''$ in such a way that $\biglor_n\Phi_n\restr_{L''_{x''y''}(\emptyset)}$ is a normal form (and consequently, $E\restr_{L''_{x''y''}(\emptyset)}$ is an equivalence relation).
\end{prop}
\begin{proof}
First, we may assume that $x',y'$ are symmetric (so that each $\Phi_n\restr_{L'_{x'y'}(\emptyset)}$ is symmetric). Now, for any $n<m$ and any formula $\varphi(x,y)\in \Phi_m(x,y)\restr_{L'_{x'y'}(\emptyset)}$ there are some formulas $\varphi_1(x,y),\varphi_2(x,y)$ in $\Phi_n(x,y)$ and $\Phi_{m-n}(x,y)$, respectively, which witness triangle inequality, that is $\models \varphi_1(x,z)\land\varphi_2(z,y)\limplies \varphi(x,y)$, and we can add to $L'$ all the symbols and to $x',y'$ all the variables from $\varphi_1,\varphi_2$ (preserving symmetry).

For each pair $n<m$ and formula $\varphi\in \Phi_m(x,y)\restr_{L'_{x'y'}(\emptyset)}$ we add these finitely many symbols and variables (adding no more than $\lvert L'\rvert+\lvert x'y'\rvert+\aleph_0$ of them at this step), and we repeat this procedure recursively countably many times (adding no more than $\aleph_0\cdot (\lvert L'\rvert+\lvert x'y'\rvert+\aleph_0)=\lvert L'\rvert+\lvert x'y'\rvert+\aleph_0$ in total). In the end, we have witnesses for all formulas.
\end{proof}

The following result is a theorem from \autocite{KMS14}, with slightly extended conclusion (which is a part of the proof there).
\begin{thm}[{\autocite[Theorem 2.5]{KMS14}}]
\label{thm:dtmtoolu}
Suppose that $X$ is a regular topological space, $\langle R_n\mid n\in \N\rangle$ is a sequence of $F_\sigma$ subsets of $X^2$, $\Sigma$ is a group of homeomorphisms of $X$, and $\mathcal O\subseteq X$ is an orbit of $\Sigma$ with the property that for all $n\in \N$ and open sets $U\subseteq X$ intersecting $\mathcal O$, there are distinct $x,y\in\mathcal O\cap U$ with $\mathcal O\cap (R_n)_x\cap (R_n)_y=\emptyset$. If $X$ is strong Choquet over $\mathcal O$, then there is a function $\tilde\phi\colon 2^{<\omega}\to \mathcal P(X)$ such that for any $\eta\in 2^\omega$ and any $n\in \omega$:
\begin{itemize}
\item
$\tilde\phi(\eta\restr n)$ is a nonempty open set,
\item
$\overline{\tilde\phi(\eta\restr{(n+1)})}\subseteq \tilde\phi(\eta\restr{n})$
\end{itemize}
Moreover, $\phi(\eta)=\bigcap_n \tilde\phi(\eta\restr n)=\bigcap_n \overline{\tilde\phi(\eta\restr n)}$ is a nonempty closed $G_\delta$ set such that for any $\eta,\eta'\in 2^\omega$ and $n\in\omega$:
\begin{itemize}
\item
if $\eta \EZ \eta'$, then there is some $\sigma\in\Sigma$ such that $\sigma\cdot \phi(\eta)=\phi(\eta')$,
\item
if $\eta(n)\neq \eta'(n)$, then $\phi(\eta)\times \phi(\eta')\cap R_n=\emptyset$, and if $\eta,\eta'$ are not $\EZ$-related, then $\phi(\eta)\times \phi(\eta')\cap \bigcup R_n=\emptyset$.
\end{itemize}
\end{thm}
\begin{proof}
As in \autocite{KMS14}: what we call $\tilde\phi(\sigma)$ here is $\gamma_{\sigma}\cdot X_{\lvert\sigma\rvert}$ in the proof there.
\end{proof}

We introduce the notion of sub-Vietoris topology which will be crucial in the application in the last section of the paper.

\begin{dfn}
Suppose $X$ is a topological space. Then by the {\em sub-Vietoris topology} we mean the topology on $\mathcal P(X)$ (i.e.\ on the family of all subsets of $X$), or on any subfamily of $\mathcal P(X)$, generated by subbasis of open sets of the form $\{A\subseteq X\mid A\cap K=\emptyset\}$ for $K\subseteq X$ closed.
\end{dfn}

As the name suggests, the sub-Vietoris topology is weaker than Vietoris topology (it differs in that the sets of the form $\{A\subseteq X\mid A\cap U\neq \emptyset \}$ for open $U$ are not included in the subbasis), and it is not, in general, Hausdorff, even when restricted to compact sets. However, we can find some spaces on which it is actually Hausdorff, e.g.\ ones as in the next simple fact.

\begin{prop}
\label{prop:taut2}
Suppose $X$ is a normal topological space (e.g.\ a compact Hausdorff space) and $\mathcal A$ is any family of pairwise disjoint, nonempty closed subsets of $X$. Then $\mathcal A$ is Hausdorff with sub-Vietoris topology.$\qed$
\end{prop}

\begin{cor}[Based on {\autocite[Corollary 2.6]{KMS14}}]
\label{cor:toolu}
Let $T$ be any first order theory with language $L$, $M$ a small model, $E$ a bounded, $F_\sigma$ equivalence relation on a type-definable subset $X$ of a product of sorts compatible with a tuple of variables $x$. Suppose $E$ has a normal form $\biglor_n \Phi_n(x,y)$. Let $Y$ be an $E$-saturated subset of $X$. Finally, suppose we have:
\begin{enumerate}
\item
some $p\in Y_M$,
\item
a countable $L'\subseteq L$, a countable $M'\preceq M\restr_{L'}$ and a countable tuple of variables $x'\subseteq x$,
\item
a group $\Sigma$ of automorphisms preserving $M$, $M'$ and all $E$-classes in $Y$ setwise.
\end{enumerate}
Such that:
\begin{enumerate}
\item
the restriction $\biglor_n\Phi_n\restr_{L'_{x'y'}(\emptyset)}(x,y)$ is a normal form (i.e.\ satisfies the triangle inequality),
\item
the topology induced on $Y_M$ by $L'$-formulas with free variables $x'$ and parameters from $M'$ is strong Choquet over $\Sigma\cdot p$
\item
For every open set $U\ni X$ in the induced topology and for all $N\in \N$, there are some $\sigma\in \Sigma$ such that $\sigma(p)\in U$ and, letting $p'=p\restr_{L'_{x'}(M')}$, we have $N<d'_{M'}(\sigma(p'),p')$.
\end{enumerate}

Then there are maps $\tilde\phi, \phi$ into $\mathcal P(Y_M)$ as in Theorem~\ref{thm:dtmtoolu} with $R_n=\{(p,q)\in (Y_M)^2\mid d_M(p,q)\leq n\}$.

(Since $R_n$ contain the diagonal, it follows that $\phi$ maps distinct points to disjoint sets, and if $\eta_1,\eta_2$ are $\EZ$-inequivalent, then $\phi(\eta_1)\times \phi(\eta_2)\cap E^M=\emptyset$.)

Furthermore, if $Y$ is pseudo-closed, then $\phi$ is a homeomorphism onto a compact subspace of $\mathcal P(Y_M)$ with sub-Vietoris topology.
\end{cor}
\begin{proof}
The first part is the same as in \autocite{KMS14}; note that thanks to the added condition about the restriction of $\biglor_n\Phi_n$ remaining a normal form, the restricted normal form gives us a metric $d'$ such that $d'\leq d$ (where $d$ is the metric obtained from the original normal form).

The ``furthermore'' part follows from the fact that $\phi$ maps distinct points onto disjoint, closed, nonempty subsets of $Y_M$, so in particular, the range $\rng(\phi)$ is a family of disjoint, closed, nonempty subsets of $Y_M$. 
Since $Y_M$ is compact  (by Corollary~\ref{cor:satinv} and Corollary~\ref{cor:bcomp}), $\rng(\phi)$ is Hausdorff with sub-Vietoris topology (by Proposition~\ref{prop:taut2}). We also see that $\phi$ is injective, so it is a bijection onto $\rng(\phi)$, and since $2^{\N}$ is compact, it is enough to show that $\phi$ is continuous.

To see that, consider a subbasic open set $U=\{F\mid F\cap K=\emptyset\}$, and notice that by compactness, $\phi(\eta)\in U$ if and only if for some $n$ we have $\overline{\tilde\phi(\eta\restr n)}\cap K=\emptyset$, which is clearly an open condition about $\eta$.
\end{proof}

\begin{thm}
\label{thm:mainAu}
Let $T$ be a complete first-order theory, $E$ an orbital on types, bounded, $F_\sigma$ equivalence relation on a $\lambda$-supported type-definable set $X$.

Take some $E$-saturated and pseudo-closed $Y\subseteq X$, and assume that there is an element $a\in Y$ whose $E$-class has infinite diameter with respect to some normal form of $E$, and choose a group $\Gamma$ witnessing that $E\restr_{[a]_\equiv}$ is orbital.

Then there is a model $M$ of size $\lvert T\rvert+\lambda$ and a function $\phi\colon 2^{\N}\to \mathcal P(Y_M\cap [a]_\equiv)$ as in the conclusion of Corollary~\ref{cor:toolu}, i.e.\ $\phi$ is a homeomorphic embedding (into $\mathcal P(Y_M\cap [a]_\equiv)$ with sub-Vietoris topology) such that for any $\eta,\eta'\in 2^{\N}$:
\begin{enumerate}
\item
$\phi(\eta)$ is a nonempty closed $G_\delta$,
\item
if $\eta\EZ\eta'$, then there exists some $\gamma\in\Gamma$ fixing $M$ setwise such that $\gamma\cdot \phi(\eta)=\phi(\eta')$ (so in particular, their saturations $[\phi(\eta)]_{E^M},[\phi(\eta')]_{E^M}$ are equal),
\item
if $\eta\neq \eta'$, then $\phi(\eta),\phi(\eta')$ are disjoint,
\item
if $\eta,\eta'$ are not $\EZ$-related, then $\phi(\eta)\times\phi(\eta')\cap E^M=\emptyset$.
\end{enumerate}
\end{thm}
\begin{proof}
We modify the proof of Theorem 5.1 of \autocite{KMS14} in a similar way to how we modified Theorem 4.13 there to  prove Theorem~\ref{thm:mainA}.

First, we may assume without loss of generality that $Y\subseteq [a]_\equiv$ (by replacing it with the intersection) and that $X=[a]_\equiv$; then $E$ is orbital as witnessed by some $\Gamma$.

Then, we may take a normal form $\biglor_n\Phi_n(x,y)$ as in Proposition~\ref{prop:nfrm}, so that the induced metric satisfies $d\leq d_L$ and $[a]_E$ is not $d$-bounded.

Finally, we want to satisfy the assumptions of Corollary~\ref{cor:toolu}, which is done in a manner analogous to Theorem 5.1 of \autocite{KMS14}: the only difference is that we need to make sure that the restriction of $\biglor_n\Phi_n(x,y)$ is still a normal form, but for that we just need to add another step to the construction to make sure we have all the witnesses (i.e.\ symbols of $L$ and variables to express necessary formulas) for triangle inequality (like we did in the proof of Proposition~\ref{prop:nfrst}).
\end{proof}

\begin{cor}\label{cor:mainGu}
Suppose that $G$ is a definable group and $H\unlhd G$ is an invariant, normal subgroup of bounded index, which is $F_\sigma$ (equivalently, generated by a countable family of type-definable sets). Suppose in addition that $K\geq H$ is a pseudo-closed subgroup of $G$. Then, if $H$ is not type-definable, then there is a small model $M$ and homeomorphic embedding $\phi\colon 2^{\N}\to\mathcal P(K_M)$ 
(where $\mathcal P(K_M)$ is equipped with sub-Vietoris topology) such that for any $\eta,\eta'\in 2^{\N}$:
\begin{enumerate}
\item
$\phi(\eta)$ is a nonempty closed $G_\delta$,
\item
if $\eta\EZ\eta'$, then the saturations $[\phi(\eta)]_{E_H^M}$ and $[\phi(\eta')]_{E_H^M}$ are equal,
\item
if $\eta\neq \eta'$, then $\phi(\eta),\phi(\eta')$ are disjoint,
\item
if $\eta,\eta'$ are not $\EZ$-related, then the saturations $[\phi(\eta)]_{E_H^M}$ and $[\phi(\eta')]_{E_H^M}$ are disjoint.
\end{enumerate}
\end{cor}
\begin{proof}
Analogous to Corollary~\ref{cor:mainG}, only instead of Theorem~\ref{thm:mainA} we use Theorem~\ref{thm:mainAu}: we get $\phi$ for $E_{H,X}$ and we compose it with the homeomorphism from Proposition~\ref{prop:grres}.
\end{proof}

\section{Characterisation of smooth equivalence relations and Borel cardinalities}
In this section, we will attempt to characterise the bounded, orbital on types and $F_\sigma$ equivalence relations which are smooth, and in particular, compare smoothness and type-definability. Throughout this section, we will assume that the language is countable, along with all the small models and supports of considered equivalence relations (so that the relevant type spaces are Polish).

Firstly, we analyse several examples showing us some of the limitations of this attempt.

\subsection{Counterexamples}

\begin{prop}
\label{prop:infdm}
Suppose $E$ is a type-definable equivalence relation on a type-definable set $X$, and that there are countably many complete $\emptyset$-types on $X$, and infinitely many of them are not covered by singleton $E$-classes. Then $E$ has a normal form such that the classes of $E$ have unbounded diameter (that is, there is no uniform bound on the diameter).
\end{prop}
\begin{proof}
Let $p_n$ with $n>0$ be an enumeration of complete $\emptyset$-types on $X$. Then put (for $n>0$)
\[
\Phi_n(x,y)=(x=y)\lor \left(E\land \biglor_{m_1,m_2\leq n} p_{m_1}(x)\land p_{m_2}(y)\right).
\]
It is easy to see that for each $n$, $\Phi_n(x,y)$ is a type-definable equivalence relation and $\Phi_n$ is increasing, so $\biglor_n\Phi_n(x,y)$ is trivially a normal form. In addition, any non-singleton $E$-class intersecting $p_n$ has diameter at least $n+1$.

There are infinitely many $p_n$ which intersect an $E$-class which is not a singleton, so in particular, the non-singleton classes have no (finite) uniform bound on diameter.
\end{proof}

\begin{ex}
\label{ex:acfdiam}
Let $T=ACF_0$ be the theory of algebraically closed fields of characteristic $0$. Consider $E={\equiv_{\KP}}$ as a relation on $\mon^2$. The space $S_2({\bf Q}^{\textrm{alg}})$ is countable, because $T$ is $\omega$-stable and ${\bf Q}^{\textrm{alg}}$ is a countable model. 
This also implies that $\equiv_{\KP}$ has only countably many classes (on the set of pairs).
It is also, of course, smooth, orbital and even type-definable.

Despite being rather well-behaved, $E$ still has a normal form with respect to which the classes have arbitrarily large diameter, which can be seen as follows.

(The set of realisations of) each type of the form $\tp(q,t/\emptyset)$ with $q\in {\bf Q}$ and $t$ transcendental is a single, infinite ${\equiv_{\KP}}$-class (because it is the set of realisations of a single type over ${\bf Q}^{\textrm{alg}}$), and in particular, it is not covered by singleton classes.
Furthermore, $S_2(\emptyset)$ is countable (because $T$ is $\omega$-stable). Therefore, by Proposition~\ref{prop:infdm}, $E$ has a normal form with respect to which its classes have arbitrarily large diameter.
\end{ex}

\begin{prop}
Suppose there is a non-isolated complete $\emptyset$-type $p_0$ such that $p_0(\mon)$ is not contained in a single class of some definable, bounded (equivalently, with finitely many classes) equivalence relation $E$. Then the relation
\[
E'(x,y)=(E(x,y)\lor \neg p_0(x))\land (x\equiv y)
\]
is $F_\sigma$ and smooth, but not type-definable.

Furthermore, if $E\cap {\equiv}$ is orbital on types, then so is $E'$.
\end{prop}
\begin{proof}
A definable and bounded equivalence relation has only finitely many classes, so $E'$ differs from ${\equiv}$ only in that one class of $\equiv$ (namely $p_0(\mon)$) is divided into finitely many pieces. Fix a countable model $M$ and a Borel reduction $f\colon S(M)\to X$ of ${\equiv}^M$ as an equivalence relation on $S(M)$ to $\Delta(X)$, equality on a Polish space $X$ (which exists because $\equiv$ is smooth, being type-definable).

Let $[p_0]_{\equiv}/E=\{ A_1,\ldots,A_n\}$. Then define $\tilde f\colon S(M)\to X\sqcup \{1,\ldots,n\}$ (where $\sqcup$ is the disjoint union and $\{1,\ldots,n\}$ has discrete topology) by
\[
\tilde f(\tp(x/M))=
\begin{cases}
f(\tp(x/M)) & \mbox{if}\;\, x\not\models p_0, \\
j & \mbox{if}\;\, x\in A_j.
\end{cases}
\]
Then clearly $\tilde f$ is Borel and witnesses that $E'$ is smooth.

$E'$ is easily seen to be $F_\sigma$, as it is the intersection of the open (and therefore $F_\sigma$, as the language is countable) set $(E(x,y)\lor \neg p_0(x))$ and the closed set $(x\equiv y)$.

It remains to show that $E'$ is not type-definable. For that, we need the following
\begin{clm}
For any formula (without parameters) $\psi\in p_0$, there is some $x\models p_0$ and $x'\not\models p_0$ such that $x'\models \psi$ and $E(x,x')$. In fact, we can find such $x'$ for any $x\models p_0$.
\end{clm}
\begin{clmproof}
The proof is by contraposition: we assume that there are no such $x,x'$ for $\psi$, and we will show that $p_0$ is isolated. Let 
\[
E''(x,y)= (E(x,y)\land \psi(x)\land \psi(y))\lor (\neg\psi(x)\land \neg \psi(y)).
\]
Then $E''$ is a definable equivalence relation which has finitely many classes (at most $1$ more than $E$) and (by the assumption), $p_0(\mon)$ is a union of $E''$-classes, of which there are only finitely many, so $p_0(\mon)$ is definable with some parameters. But since it is invariant, it implies that it is definable without parameters, and therefore $p_0$ is isolated.

Once we have some $x'$ for a single $x\models p_0$, we may obtain one for each of them simply by applying automorphisms.
\end{clmproof}
Now we choose a sequence $\varphi_n$ of formulas such that $\bigland_n\varphi_n\proves p_0$ and $\varphi_{n+1}\proves \varphi_n$. Let $x_0,y_0\models p_0$ be such that $\neg E(x_0,y_0)$ (which we can find because $p_0$ is not contained in a single $E$-class), and let $x_n$ be a sequence of elements satisfying $\varphi_n$ but not $p_0$, and simultaneously satisfying $E(x_n,x_0)$ (this sequence exists by the claim), and let $y_n$ be a sequence such that each $(x_0,x_n)$ is conjugate to $(y_0,y_n)$ (so that $x_n\equiv y_n$ and $y_n\models \varphi_n$ and $E(y_0,y_n)$).

Then any limit point of the sequence $\tp(x_n,y_n/\emptyset)$ in $S_2(\emptyset)$ is not in $E'$, even though each $\tp(x_n,y_n/\emptyset)$ is in $E'$, so $E'$ is not type-definable.

The ``furthermore'' part is obvious, since $E'$ agrees with $E\cap{\equiv}$ on $p_0$ and is total when restricted to any other type.
\end{proof}

\begin{ex}
\label{ex:ntsmt}
Consider $T=\operatorname{Th}(\Z,+)$ (the theory of additive group of integers) and the type $p_0=\tp(1/\emptyset)$ (the type of an element not divisible by any natural number).

The type $p_0$ is not isolated, and it is not contained in a single class of the definable relation $E$ of equivalence modulo $3$, while $E\cap {\equiv}$ has at most two classes in each complete type, so it is orbital on types due to Corollary~\ref{cor:2cl}.

In particular -- by the preceding proposition -- the relation $E'(x,y)$ which says that $x\equiv y$ and they either have the same residue modulo $3$ or else each of them is divisible by some natural number (i.e.\ they are not of the same type as $1$), is $F_\sigma$, orbital on types and smooth, but not type-definable.
\end{ex}

\begin{ex}
\label{ex:prbfin}
Suppose there is some $a\in \mon$ such that $\equiv_{\KP}$ has two classes on $[a]_{\equiv}$ (like $a=\sqrt 2$ for $\mon\models ACF_0$), so that ${\equiv_{\KP}}\restr_{[a]_{\equiv}}\sim_B \Delta(2)$. Consider the infinite disjoint union of copies of $\mon$, i.e.\ the multi-sorted structure $(\mon_n)_{n\in\N}$ where each $\mon_n$ is a distinct sort isomorphic to $\mon$ (without any relations between elements of $\mon_n$ and $\mon_m$ for $n\neq m$). Then consider $\overline{a}=(a_n)_{n\in \N}$ where $a_n$ is the element of $\mon_n$ corresponding to $a$. Then $[\overline a]_{\equiv}=\prod_n [a_n]_\equiv$ and similarly \[
(b_n)_n\equiv_{\KP} (c_n)_n\iff \bigland_n b_n\equiv_{\KP} c_n
\]
(by \autocite[Lemma 3.7(iii)]{CLPZ01}). Now consider the following relation $E$ on $[\overline a]_{\equiv}$:
\[
(b_n)_n \Er (c_n)_n\iff \{n\mid b_n \not\equiv_{\KP} c_n\} \textrm{ is finite}.
\]
Then $E$ is refined by $\equiv_{\KP}$, but on the other hand, 
\[
E\sim_B ({\equiv_{\KP}}\restr_{[a]_{\equiv}})^\N/\textrm{Fin}\sim_B \Delta(2)^\N/\textrm{Fin}={\EZ}.
\]
(This can be seen e.g.\ by considering the equivalence relation on $[\overline{a}]_{\equiv}/{\equiv_{\KP}}$ induced by $E$, which is easily seen to be bireducible with $E^M$ for a countable model $M$, using Lemma~\ref{lem:cons} and Proposition~\ref{prop:eqcmp}.)

In particular, $E$ is not smooth, it is easy to see that $E$ is $F_\sigma$ (because $\equiv_{\KP}$ is type-definable and there are countably many finite subsets of $\N$), and it is also orbital, as its classes are just the orbits of the group
\[
\left\{(\sigma_n)_n\in \prod_{n\in\N}\Aut(\mon_n) \mathrel{}\middle|\mathrel{} \textrm{for all but finitely many }n,\, \sigma_n\in \Autf{KP}(\mon_n)\right\}.
\]
Additionally, $E$ is only defined on a single type and is not type-definable, so by Corollary~\ref{cor:maincont2}, all its classes have infinite diameter.
\end{ex}

\begin{ex}
\label{ex:nreft}
Consider a saturated model $K$ of the theory $T=\Th (\R,+,\cdot,0,1,<)$ of real closed fields. For each $n\in {\bf N^+}$ we have a type-definable equivalence relation $\Phi_n(x,y)=\bigland_{k\geq n} (x<k\liff y<k)$. Consider the relation $E=\biglor_n \Phi_n$ (with $\Phi_0(x,y)=\{x=y\}$, as before):
\begin{itemize}
\item
$E$ is an $F_\sigma$ equivalence relation (and since $\Phi_n$ is an increasing sequence of equivalence relations, it is easy to see that $\biglor_n \Phi_n$ is its normal form).	 
\item
E has two classes: the class $C_{\textrm{fin}}$ of elements bounded from above by some natural number, and its complement $C_{\infty}$. Therefore, it is bounded and smooth.
\item
$C_{\textrm{fin}}$ is a class which is not pseudo-closed (otherwise, by compactness, it would intersect $\bigland_n x>n=C_{\infty}$).
\end{itemize}
This combination of features is possible because $E$ does not refine $\equiv$ (and therefore it is not orbital on types), so we cannot apply Theorem~\ref{thm:mainA} to it.
\end{ex}

\begin{ex}[{\autocite[Example 3.38]{KM14}}]
\label{ex:simon}
Let $T$ be the theory of an infinite dimensional vector space over ${\bf F}_2$ in the language $(+,0,U_n)_{n\in \N}$ (i.e.\ an infinite abelian group of exponent $2$), where $U_n$ are predicates for independent subspaces of codimension $1$ (i.e.\ subgroups of index $2$).

Consider $G=\mon\models T$ as a definable (additive) group, and let $H\leq G$ be the intersection of all $U_n$. Then $[G:H]=\mathfrak c$, and cosets of $H$ are exactly the types $X_\eta=\bigcap_nU_n^{\eta_n}$, where $\eta\colon \N\to \{0,1\}$, while $U_n^0=U_n$ and $U_n^1=\mon\setminus U_n$.

Consider the subspaces $W_\theta\leq G$ defined as $W_\theta=\pi^{-1}[\ker(\theta)]$, where $\pi\colon G\to G/H$ is the quotient map, and $\theta$ is a nonzero functional $G/H\to {\bf F}_2$. Each $\theta$ is uniquely determined by $W_\theta$ (since its value is $0$ on $\pi[W_\theta]$ and $1$ elsewhere and $\ker\pi$ is contained in all $W_\theta$), so there are $\lvert (G/H)^*\rvert=\mathfrak c >\aleph_0$ distinct $W_\theta$, in particular some $W=W_\theta$ is not definable.

On the other hand, $W$ is invariant, as it is the union of some $X_\eta$ which are type-definable, and $[G:W]=2$ (because $W$ has codimension $1$), so $W$ is not type-definable (if it was, its complement would also be type-definable, as it is invariant and a coset of $W$).

Let us expand $\mon$ to $(\mon,\mathfrak X,\cdot)$, where $\mathfrak X$ is a principal homogeneous space for $G$. By Proposition~\ref{prop:grrel}, $W$ induces an invariant equivalence relation $E_{W,\mathfrak X}$ on $\mathfrak X$ which has two classes, is orbital by Proposition~\ref{prop:grnrm} (or Corollary~\ref{cor:2cl}) and not type-definable by Proposition~\ref{prop:grres}.
\end{ex}

\begin{rem}
The equivalence relation in the previous example is not type-definable, and it is unlikely to even be Borel, as the subspace $W$ is the kernel of an almost arbitrary linear functional, which can be very ``wild''. It does show, however, that we need some ``definability'' hypotheses beyond invariance for the likes of Theorem~\ref{thm:twN}. 
\end{rem}

\subsection{Main characterisation theorem}

\begin{thm}
\label{thm:mainB}
Let $E$ be an $F_\sigma$, bounded, orbital on types equivalence relation on a type-definable set $X$, and $d$ be the invariant metric induced by a normal form of $E$. Then consider the four conditions:
\begin{enumerate}
\item
\label{it:bdd}
$E$ classes have uniformly bounded diameter with respect to $d$.
\item
\label{it:tdf}
$E$ is type-definable.
\item
\label{it:smt}
$E$ is smooth.
\item
\label{it:fnd}
$E$ classes have finite diameter (equivalently, by Corollary~\ref{cor:diam}, they are pseudo-closed).
\end{enumerate}
These conditions are related as follows:
\begin{itemize}
\item
\eqref{it:bdd} implies \eqref{it:tdf},\eqref{it:smt},\eqref{it:fnd},
\item
\eqref{it:tdf} implies \eqref{it:smt} and \eqref{it:fnd}, but not \eqref{it:bdd}
\item
\eqref{it:smt} is equivalent to \eqref{it:fnd}, but it does not imply \eqref{it:tdf} or \eqref{it:bdd}.
\end{itemize}

If we assume, in addition, that $E$ refines ${\equiv_{\KP}}$ (on $X$), then conditions \eqref{it:tdf},\eqref{it:smt},\eqref{it:fnd} are equivalent (and equivalent to simply $E={\equiv_{\KP}}\restr_X$) and are implied by, but do not imply \eqref{it:bdd}.

If we assume instead that $E$ is only defined on a single complete $\emptyset$-type, then all conditions are equivalent.
\end{thm}
\begin{proof}
For the first part:
\begin{itemize}
\item
\eqref{it:bdd} trivially implies \eqref{it:tdf}
\item
That \eqref{it:tdf} implies \eqref{it:smt} follows from Fact~\ref{fct:tdsmt}.
\item
That \eqref{it:smt} implies \eqref{it:fnd} follows from (the contraposition of) Theorem~\ref{thm:mainA} with $Y=X$.
\item
That \eqref{it:fnd} implies \eqref{it:smt} follows from the fact that in this case, classes of $E^M$ are closed. Indeed, if $E^M$ was not smooth, then the Harrington-Kechris-Louveau dichotomy would imply that there exists a continuous reduction of $\EZ$ to $E^M$, which would contradict the fact that classes of $\EZ$ are not closed.  (This was pointed out to the authors by Itay Kaplan after circulation of a preprint of this paper.)
\item
That \eqref{it:tdf} does not imply \eqref{it:bdd} follows from Example~\ref{ex:acfdiam}.
\item
That \eqref{it:smt} does not imply \eqref{it:tdf} is demonstrated by Example~\ref{ex:ntsmt}.
\item
Other listed implications (or lack thereof) are logical consequences of the ones above.
\end{itemize}

To show that the last three conditions are equivalent if $E$ refines $KP$-type, it is enough to show that \eqref{it:fnd} implies \eqref{it:tdf}. But it follows easily from Corollary~\ref{cor:diam}. That it does not imply \eqref{it:bdd} can be seen in Example~\ref{ex:acfdiam}.

To show that all four are equivalent if $E$ is defined on a single type, it is enough to notice that \eqref{it:fnd} implies \eqref{it:bdd}. But this is immediate from the fact that all classes have the same diameter (by Fact~\ref{fct:stdiam}).
\end{proof}

\begin{rems} $\,$
\begin{itemize}
\item
We have in particular, for bounded $E$ which are $F_\sigma$, orbital on types and either defined on a single complete $\emptyset$-type, or refining ${\equiv_{\KP}}$, that $E$ is type-definable if and only if it is smooth.
\item
The property that $E$ has only countably many classes implies \eqref{it:tdf},\eqref{it:smt},\eqref{it:fnd}, but not \eqref{it:bdd} (and is not implied by any of the conditions). \eqref{it:tdf} follows immediately from Theorem~\ref{thm:twN}, while the others follow as a consequence of Theorem~\ref{thm:mainB}. That having countably many classes does not imply \eqref{it:bdd}, we have seen in Example~\ref{ex:acfdiam}. That none of the conditions imply that there is only a countable number of classes can be seen by examining $\equiv$ in a non-small theory.
\item
Example~\ref{ex:prbfin} along with Corollary~\ref{cor:diam} show that the condition that $E$ is refined by $\equiv_{\KP}$ is strictly weaker than all the conditions in Theorem~\ref{thm:mainB}, even with the added assumption that $E$ is only defined on a single type. (Of course, it is not \emph{strictly} weaker if we assume that $E$ refines $\equiv_{\KP}$, although trivially so.)
\end{itemize}
\end{rems}

\subsection{Possible extensions of the characterisation theorem}
If, in the first part of Theorem~\ref{thm:mainB}, we drop the assumption that $E$ is orbital on types (so we allow $E$ to not refine type), then \eqref{it:tdf} does not imply 
\eqref{it:fnd} -- as witnessed by Example~\ref{ex:nonf} (though \eqref{it:bdd} certainly still implies the other conditions and \eqref{it:tdf} implies \eqref{it:smt} and that the classes are pseudo-closed).

We can, however, replace the assumption that $E$ is {\it orbital on types} by the assumption that it {\it refines type}.
(Note that by Example~\ref{ex:finperm} and Example~\ref{ex:trace}, this is not a trivial replacement.)
In this case, we do not know whether \eqref{it:smt} implies \eqref{it:fnd} (even if we assume that $E$ refines $\equiv_{\KP}$ or is defined on a single type),
whereas the other implications hold as before.

If we drop the requirement that $E$ is $F_\sigma$, points \eqref{it:bdd} and \eqref{it:fnd} do not make sense. However, the other two do, so another question that arises naturally is the following.

\begin{qu}
Suppose that $E$ is a Borel, bounded equivalence relation which is defined on a single complete $\emptyset$-type or which refines $\equiv_{\KP}$. 
Is it true that smoothness of $E$ implies that $E$ is type-definable?
\end{qu}

In light of Theorem~\ref{thm:mainB},  the assumption that $E$ is defined on a single complete $\emptyset$-type or that $E$ refines $\equiv_{\KP}$ in the above question is natural -- as we have seen in Theorem~\ref{thm:mainB}, without this assumption, even for $F_\sigma$-relations, smoothness does not imply type-definability.
Note also that we need at least some weak ``definability'' assumptions, as illustrated by Example~\ref{ex:simon}. Borelness seems the most natural such an assumption, but it is possible that one should assume something stronger. On the other hand, it is conceivable that some weaker assumptions than Borelness would suffice. Maybe one should also add the assumption that the relation in question is orbital on types, which is present in Theorem~\ref{thm:mainB}.

The methods used in this paper are not suitable to deal with the above question (in the case when the relation is not $F_\sigma$), because in this context, we do not have any natural notion of a distance or diameter.

\section{Applications to definable group extensions}
\subsection{Introduction to extensions by abelian groups}

This section will show an important application of Corollary~\ref{cor:mainGu} to definable extensions by abelian groups. More specifically, we deal with short exact sequences of groups of the form
\begin{equation}
\label{eq:shexs}\tag{$\dagger$}
0\to A\to \wt G\to G\to 0,
\end{equation}
where $A$ is an abelian group. In this case, there is a full algebraic description of $\wt G$ in terms of an action of $G$ on $A$ by automorphisms (induced by conjugation in $\wt G$) and a (2-)cocycle $h\colon G^2\to A$ (to be defined shortly). We define multiplication on $A\times G$ by the formula
\begin{equation}
\label{eq:mext}
\tag{$\dagger\dagger$}
(a_1,g_1)\cdot (a_2,g_2)=(a_1+g_1\cdot a_2+h(g_1,g_2),g_1g_2).
\end{equation}
And a cocycle is defined as follows.
\begin{dfn}
Let $G$ be a group acting on an abelian group $A$. A function $h\colon G^2\to A$ is a \emph{2-cocycle} if it satisfies, for all $g,g_1,g_2,g_3\in G$, the following equations:
\begin{align*}
h(g_1,g_2)+h(g_1g_2,g_3)&=h(g_1,g_2g_3)+g_1\cdot h(g_2,g_3),\\
h(g,e)=h(e,g)&=e.
\end{align*}
\end{dfn}

The equation \eqref{eq:mext} endows $A\times G$ with group structure -- with inverse $(a,g)^{-1}=(-g^{-1}\cdot a-h(g^{-1},g),g^{-1})$ -- which is compatible with the exact sequence \eqref{eq:shexs}, and any $\wt G$ in such a short exact sequence has this form. In this language, the properties of the action and of $h$ reflect the properties of the extension, e.g.\ central extensions correspond to trivial actions of $G$ on $A$. More information about this subject (in abstract algebraic terms) can be found in e.g.\ \autocite{Rotman} (section 10.3., and in particular the part up to and including Theorem 10.14).

\begin{dfn}
We work in an arbitrary given structure.
A \emph{definable extension of a definable group $G$ by a definable abelian group $A$} is a tuple $(G,A,*,h)$, where $*$ is a definable action of $G$ on $A$ by automorphisms and $h\colon G^2\to H$ is a definable 2-cocycle. We will also call that the group $\wt G=A\times G$ with multiplication defined as in \eqref{eq:mext}.
\end{dfn}

\begin{rem}
The group $\wt G$ introduced above is definable.
\end{rem}

In \autocite{GK13}, the authors have shown that such extensions can, under some additional assumptions, give new examples of definable groups with $G^{00}\neq G^{000}$, building upon and extending the intuitions from the first known example with this property, found in \autocite{CP12}, namely the universal cover of $SL_2(\R)$. They also pose some questions and conjectures, one of which will be proved at the end of this section. To state their main result, we need the following definition.

\begin{dfn}
A 2-cocycle $h\colon G^2\to A$ is \emph{split via $f\colon G\to A$} if for all $g_1,g_2\in G$ we have
\begin{align*}
h(g_1,g_2)&=df(g_1,g_2):=f(g_1)+g_1\cdot f(g_2)-f(g_1g_2),\\
f(e)&=0.
\end{align*}
\end{dfn}

Now, we recall the main theorem from \autocite{GK13} (namely, Theorem 2.2 from there). Later in this section, we will recall and prove two equivalent conjectures from \autocite{GK13} which imply that the main technical assumption in this theorem (i.e.\ assumption (i) below) is not only a sufficient (together with (ii)) but also a necessary condition in a rather general situation. 
\begin{thm}
\label{thm:twrgk}
Let $G$ be a group acting by automorphisms on an abelian group $A$, where $G$, $A$ and the action of $G$ on $A$ are $\emptyset$-definable in a (saturated) structure $\G$, and let $h\colon G \times G \to A$ be a 2-cocycle which is $B$-definable in $\G$ and with finite range $\rng(h)$ contained in $\dcl(B)$ (the definable closure of $B$) for some finite parameter set $B\subset \G$. By $A_0$ we denote the subgroup of $A$ generated by $\rng(h)$. Additionally, let $A_1$ be a bounded index subgroup of $A$ which is type-definable over $B$ and which is invariant under the action of $G$.
Finally, suppose that:
\begin{enumerate} 
\item[(i)] the induced 2-cocycle $\overline{h}_{|{G}^{00}_B \times {G}^{00}_B}\colon {G}^{00}_B \times {G}^{00}_B \to A_0/\left({A}_{1} \cap A_0\right)$ is non-split via $B$-invariant functions (i.e.\ there is no $B$-invariant function $f\colon G^{00}_B\to A_0/(A_1\cap A_0)$ such that $\overline{h}\restr_{{G}^{00}_B}$ is split via $f$),
\item[(ii)] $A_0/\left( {A}_{1} \cap A_0 \right)$ is torsion free (and so isomorphic with $\Z^n$ for some natural $n$).
\end{enumerate}
Then $\wt{G}^{000}_B \ne \wt{G}^{00}_B$ (where $\wt G$ is $G\times A$ with the group structure defined by \eqref{eq:mext}).
\end{thm}

\subsection{Main theorem for definable group extensions}
In this subsection, we will fix some (arbitrary) definable extension of a definable group $G$ by an abelian definable group $A$, the group $\wt G=A\times G$ with cocycle $h$ and the corresponding short exact sequence
\[
0\to A\to \wt G\overset{\pi} \to G\to 0.
\]

We intend to prove the following theorem, and to that end, we will use Corollary~\ref{cor:mainGu} and two other results which we will obtain soon.

\begin{thm}
\label{thm:mainext}
Suppose $\wt H\unlhd \wt G$ is invariant and of bounded index, generated by a countable family of type-definable sets, contained in a type-definable group $\overline H\leq \wt G$. Put $A_H:=\overline H\cap A$ and $H=\pi[\overline H]$ (both of these are naturally type-definable groups).

Assume that $H=\pi[\wt H]$ and that it acts trivially on $A_H/\wt H\cap A$, while $\wt H\cap A$ is type-definable. Then $\wt H$ is type definable.
\end{thm}

Until the end of the proof of Theorem~\ref{thm:mainext}, we are working the in the context (and with the notation) of this theorem.

\begin{lem}[generalization of a part of the proof of {\autocite[Proposition 2.14]{GK13}}]
\label{lem:pmap}$\,$\\
Given an assignment $g\mapsto a_g$ such that for $g\in H$ we have $(a_g,g)\in \wt H$, the formula $\Phi((a,g)\cdot \wt H)=a-a_g+(\wt H\cap A)$ yields a well-defined, injective function $\overline H/\wt H\to A_H/(\wt H\cap A)$, which does not depend on the choice of $a_g$.
\end{lem}
\begin{proof}
Consider any $(a_1,g_1),(a_2,g_2)\in \overline H$.

Since $(a_{g_1},g_1),(a_{g_2},g_2)\in \wt H$, we have
\begin{multline}
\label{eq:map}
(a_{g_1}-g_1g_2^{-1}\cdot a_{g_2}-g_1\cdot h(g_2^{-1},g_2)+h(g_1,g_2^{-1}),g_1g_2^{-1})=\\
=(a_{g_1},g_1)(-g_2^{-1}a_{g_2}-h(g_2^{-1},g_2),g_2^{-1})=(a_{g_1},g_1)(a_{g_2},g_2)^{-1}\in \wt H.
\end{multline}
We also have, by simple calculation (for arbitrary $g\in G$ and $a,a'\in A$)
\begin{equation}
\label{eq:remc}
(a,g)^{-1}(a',g)=(0,g)^{-1}(-a,e)(a',e)(0,g)=(g^{-1}\cdot(a'-a),e).
\end{equation}

We want to show that $(a_1,g_1)(a_2,g_2)^{-1}\in \wt H$ if and only if $(a_1-a_{g_1})-(a_2-a_{g_2})\in \wt H\cap A$. The first condition says that
\[
(a_1-g_1g_2^{-1}\cdot a_2-g_1\cdot h(g_2^{-1},g_2)+h(g_1,g_2^{-1}),g_1g_2^{-1})=(a_1,g_1)(a_2,g_2)^{-1}\in \wt H.
\]
Multiplying it on the left by the inverse of the LHS of \eqref{eq:map}, and applying \eqref{eq:remc}, we infer that it is equivalent to
\[
(g_1g_2^{-1})^{-1}\left((a_1-a_{g_1})-g_1g_2^{-1}\cdot(a_2-a_{g_2})\right)\in \wt H\cap A.
\]
On the other hand, since the action of $H$ on the cosets of $\wt H\cap A$ is trivial, and $g_1,g_2\in H$, we can cancel both $g_1g_2^{-1}$, and in conclusion the first condition is equivalent to
\[
(a_1-a_{g_1})-(a_2-a_{g_2})\in \wt H\cap A,
\]
which is just the second condition. Independence of the choice of $a_g$ easily follows: for any $a_1,a_2$ such that $(a_1,g),(a_2,g)\in \wt H$ we have, by the above, 
\[
a_1-a_2=(a_1-a_g)-(a_2-a_g)\in \wt H\cap A. \qedhere
\]
\end{proof}

\begin{prop}
Let $M$ be a small model. There exists a continuous function $f\colon S_{\overline H}(M)\to A_H/(A\cap \wt H)$ such that for $p,q\in S_{\overline H}(M)$ we have $p\Er_{\wt H}^M q\iff f(p)=f(q)$.
\end{prop}
\begin{proof}
$\wt H$ is $F_\sigma$, so we can write it as a union $\bigcup_i D_i$ of type-definable subsets of $\overline H$. Without loss of generality, we can assume that $D_i$ are all symmetric, contain $e$ and satisfy $D_i^2\subseteq D_{i+1}$. Now, consider $C_i:=\{g\in H\mid \exists a\,(a,g)\in D_i\}$. 
Then $C_i$ are also type-definable, symmetric, contain $e$ and satisfy $C_i^2\subseteq C_{i+1}$, and since $\pi[\wt H]=H$, there is some $n$ such that $H=C_n$ (by Corollary~\ref{cor:sumt}).

Consider the map $f_1\colon \overline H\to A_H/(A\cap \wt H)$ defined as $f_1(a,g)=a-a_g+(A\cap \wt H)$, where $a_g$ is such that $(a_g,g)\in D_n$ (which exists because $C_n=H$). This is well-defined and does not depend on the choice of $a_g$ by Lemma~\ref{lem:pmap}, which also says that $f_1(a,g)=f_1(a',g')$ if and only if $(a,g)\Er_{\wt H} (a',g')$.

Furthermore, since cosets of $A\cap \wt H$ are $M$-invariant (by Corollary~\ref{cor:satinvg}, because $A\cap \wt H$ is invariant of bounded index), this map factors through $S_{\overline H}(M)$, yielding a map $f\colon  S_{\overline H}(M)\to A_H/(\wt H\cap A)$ which has the desired properties:
\begin{itemize}
\item
$p\Er_{\wt H}^M q\iff f(p)=f(q)$, because $f_1(a,g)=f_1(a',g')\iff (a,g)\Er_{\wt H} (a',g')$,
\item
$f$ is continuous: a basic closed set in $A_H/(\wt H\cap A)$ is the quotient of an $M$-invariant pseudo-closed set, so it is of the form $\Phi(x,M)+(\wt H\cap A)$ for some type $\Phi$. But $f^{-1}[\Phi(x,M)+(A\cap \wt H)]$ is just 
\[
\{\tp(a,g/M)\mid \exists a'\,((a',g)\in D_n\land a-a'\in \Phi(x,M)+(\wt H\cap A)) \},
\]
which is clearly a closed set in $S_{\overline H}(M)$ (owing to the fact that $A\cap \wt H$ is type-definable). \qedhere
\end{itemize}
\end{proof}

Having proved the above results, we have all but finished the proof of the main theorem.
\begin{proof}[Proof of Theorem~\ref{thm:mainext}]
Suppose for a contradiction that $\wt H$ is not type-definable. Then let us choose a small model $M$ as in Corollary~\ref{cor:mainGu} for $H=\wt H$ and $K=\overline H$, so that we have a homeomorphic embedding 
$\phi\colon 2^\N\to \mathcal P(S_{\overline{H}}(M))$ (the latter with sub-Vietoris topology) such that if $\eta \EZ \eta'$,
then the saturations $[\phi(\eta)]_{E_{\wt H}^M},[\phi(\eta)]_{E_{\wt H}^M}$ are equal, and otherwise they are disjoint.

The continuous map $f$ from the previous proposition induces a continuous (in sub-Vietoris topology) mapping 
$f''\colon \mathcal P(S_{\overline{H}}(M))\to \mathcal P(A_H/(A\cap \wt H))$, and it is easy to see that $f''\circ \phi$ maps $\EZ$-related points to the same sets, while it maps unrelated points to disjoint closed subsets (because the value of $f$ depends only on the 
$E_{\wt H}^M$-class of the argument and any two points from different $E_{\wt H}^M$-classes are mapped by $f$ to distinct elements).

Consequently, the range of $f''\circ \phi$ is a compact Polish space: this is because it is Hausdorff by the compactness of $A_H/(A\cap \wt H)$ and Proposition~\ref{prop:taut2}, and the image by a continuous function of a compact Polish space in a Hausdorff space is Polish (as one can easily show that such an image is compact and second-countable). 
But then $f''\circ \phi$ is a continuous reduction of $\EZ$ to equality on $\rng(f''\circ\phi)$, which is a contradiction.
\end{proof}

\subsection{Application of the main theorem in group extensions}

In \autocite{GK13}, the authors state the following two equivalent conjectures, which we will prove using Theorem~\ref{thm:mainext}.

\begin{con}[{\autocite[Conjecture 2.11]{GK13}}]
\label{con:gk211}
Suppose we have a definable extension of a definable group $G$ by a definable abelian group $A$, corresponding to the short exact sequence
\[
0\to A\to \wt G\to G\to 0
\]
with the $2$-cocycle $h\colon G^2\to A$. Assume that $h$ has finite range contained in $dcl(\emptyset)$ and that $G^{00}=G^{000}$.

Then, the conjecture is that for any invariant $\wt H\leq\wt {G}$ of bounded index, and such that $\wt H\cap A$ is type-definable, we have
\[
\wt{G}^{00}\cap A\subseteq \wt H\cap A.
\]
\end{con}

\begin{con}[{\autocite[Conjecture 2.10]{GK13}}]
\label{con:gk210}
Assume we have $G,A,h$ as in the first paragraph of the previous conjecture. Additionally, let $A_1$ be a type-definable subgroup of $A$ of bounded index and invariant under the action of $G$.

Then, the conjecture is that
\[
\wt G^{00}\cap A\subseteq A_1 \iff \wt G^{000}\cap A\subseteq A_1.
\]
\end{con}

\begin{rem}
The above conjectures are important mostly for two reasons:
\begin{enumerate}
\item
They imply that Corollary 2.8 in \autocite{GK13} holds in general (i.e.\ also when the language is uncountable), that is: in Theorem~\ref{thm:twrgk}, if $G^{00}_B=G^{000}_B$ and $A_1^* \subseteq \wt{G}^{000}_B \cap A$, then the assumption (i) (about the non-splitting of the modified $2$-cocycle) not only implies (under the assumption (ii)), but is also necessary for $\wt G^{00}_B\neq \wt G^{000}_B$. This is explained in detail in \autocite{GK13}.
\item
They imply that that, in a rather general context, the quotient $\wt G^{00}/\wt G^{000}$ is (algebraically) isomorphic  to the quotient of a compact group by a finitely generated dense subgroup (this will be revisited at the end of this section).
\end{enumerate}
\end{rem}

\begin{rem}
The authors of \autocite{GK13} actually allow a finite parameter set $B$ over which the cocycle $h$ is definable, they calculate the connected components over this set, and they assume that $\wt H$ is $B$-invariant in Conjecture~\ref{con:gk211} and $A_1$ is type-definable over $B$ in Conjecture~\ref{con:gk210}. But we may add constants for elements of $B$ to the language, and it changes none of the properties relevant to the previous conjectures, so we assume without loss of generality that $B=\emptyset$.

We will also drop the requirement that $h$ has finite range contained in $dcl(\emptyset)$, as it is not needed for the subsequent discussion, which leaves us in the general context of Theorem~\ref{thm:mainext}, only with the additional assumption that $G^{00}=G^{000}$.
\end{rem}

So, we are working with a definable extensions $\wt G$ of a definable group $G$ by a definable abelian group $A$.
The assumption that $G^{00}=G^{000}$ allows us to make use of the following observation.

\begin{fct}[{\autocite[Remark 2.1(iii)]{GK13}}]
If $A_1$ is a type-definable, bounded index subgroup of $A$, invariant under the action of $G$ (i.e.\ normal in $\wt G$), then $G^{00}$ acts trivially on $A/A_1$.
\end{fct}
\begin{proof}
$A_1$ is $G$-invariant, so $G$ acts naturally on $A/A_1$, yielding an abstract homomorphism $f\colon G\to S(A/A_1)$ (where $S(A/A_1)$ is the abstract permutation group). Notice that
\[
\ker(f)=\{ g\in G\mid (\forall a)\,g\cdot a-a\in A_1\} 
\]
is a type-definable subgroup of $G$, and it has bounded index, because $S(A/A_1)$ is small. Therefore, it contains $G^{00}$.
\end{proof}

This leads us to the next corollary.

\begin{cor}
Suppose $\wt H\unlhd \wt G$ is an invariant subgroup of bounded index, contained in $\wt G^{00}$ and $F_\sigma$ (i.e.\ generated by a countable family of type-definable sets), while $G^{00}=G^{000}$.

Then $\wt H$ is type-definable (and therefore equal to $\wt G^{00}$) if and only if $\wt H\cap A$ is type-definable. 
\end{cor}
\begin{proof}
$\Rightarrow$ is clear. 
For $\Leftarrow$ notice that $\pi[\wt H]$ is contained in $G^{00}$, and contains $G^{000}$, so it is, in fact, equal to $G^{00}$, and, by the previous fact, it acts trivially on $A/(\wt H\cap A)$, so the result follows immediately from Theorem~\ref{thm:mainext} with $\overline H:=\wt G^{00}$.
\end{proof}

And finally, we can prove Conjecture~\ref{con:gk211}.
\begin{cor}
\label{cor:conpr}
Suppose $\wt H\leq \wt G$ is an invariant subgroup of bounded index, and that $G^{00}=G^{000}$. Suppose in addition that $\wt H\cap A\cap \wt G^{00}$ is type-definable.
Then $\wt H\cap A\supseteq \wt G^{00}\cap A$. (In particular, Conjecture~\ref{con:gk211} holds, without any assumptions on $h$ beyond definability.)
\end{cor}
\begin{proof}
We intend to apply the preceding corollary, so we need to modify $\wt H$ to satisfy its assumptions.
\begin{enumerate}
\item
We can assume without loss of generality that $\wt H$ is normal, because we can replace it with $\Core(\wt H):=\bigcap_{g\in \wt G} g\wt Hg^{-1}$: the latter is obviously normal and invariant, it has bounded index, because it contains 
${\wt G}^{000}$, and finally, since $A,\wt G^{00}$ are normal, we have the identity
\[
\Core(\wt H) \cap A\cap \wt G^{00}=\Core(\wt H \cap A \cap \wt G^{00}),
\]
so in particular, the left hand side is type-definable, and hence $\Core(\wt H)$ satisfies the assumptions of Corollary~\ref{cor:conpr} and it is contained in $\widetilde H$.
\item
Secondly, can assume without loss of generality that $\wt H\leq \wt G^{00}$, by replacing $\wt H$ with $\wt H\cap \wt G^{00}$.
\item
Thirdly, we can also assume that $\wt H$ is generated by a countable family of type-definable sets by replacing it with $(\wt H\cap A)\cdot \wt G^{000}$. 
This is a normal subgroup, because $\wt G^{000}$ is normal, and it is generated by $\wt H\cap A$ and a type-definable set generating $\wt G^{000}$ (which exists by e.g.\ \autocite[Proposition 3.4]{GN08}). Moreover, $\wt G^{000}\leq \wt H$, so $(\wt H\cap A)\cdot \wt G^{000}\leq \wt H$.
\item
Then we can apply the previous corollary to deduce that ${\wt H}=\wt G^{00}$, so trivially ${\wt H}\cap A\supseteq \wt G^{00}\cap A$.\qedhere
\end{enumerate}
\end{proof}

As was suggested in \autocite[Remark 2.16]{GK13}, Conjecture~\ref{con:gk211} implies the next corollary (formulated in a slightly more general form here). By the discussion right after \autocite[Proposition 2.14]{GK13}, this implies that whenever we are in the context of Theorem~\ref{thm:twrgk}, and additionally $G^{000}_B=G^{00}_B$ and $A_1 \subseteq {\wt G}^{000}_B \cap A$, then ${\wt G}^{00}_B/{\wt G}^{000}_B$ is  (abstractly) isomorphic to the quotient of a compact abelian group by a dense finitely generated subgroup -- which is analogous to the fact from \autocite{CP12}) that for $G$ definable in o-minimal expansions of a real closed field, $G^{00}/G^{000}$ is (abstractly) isomorphic to the quotient of a compact, abelian Lie group by a dense, finitely generated subgroup. This furthers the analogy between \autocite{GK13} and \autocite{CP12}.

\begin{cor}
Suppose that $\wt G$ is a definable extension of a definable group $G$ by an abelian group $A$, and that $G^{00}=G^{000}$.

Let $A_1\leq \wt G^{000}\cap A$ be a $G$-invariant, type-definable subgroup of $A$ of bounded index. 
Then $\left( \wt G^{000}\cap A\right)/A_1$ is dense in $\left( \wt G^{00}\cap A\right)/A_1$ (with the logic topology).
\end{cor}
\begin{proof}
Let $A_2$ be the preimage of the closure of $(\wt G^{000}\cap A)/A_1$ by the quotient map $\pi\colon A\to A/A_1$. 
Then $A_2$ is a type-definable (over $\emptyset$)
subgroup of $A$, which contains $\wt G^{000}\cap A$. 

Then, $\wt H:=A_2\cdot \wt G^{000}$ is an invariant subgroup of $\wt G$ of bounded index, and $\wt H\cap A=A_2$ is type-definable, as is $\wt H\cap A \cap \wt G^{00}$, so by Corollary~\ref{cor:conpr}, $A_2=\wt H\cap A\supseteq \wt G^{00}\cap A$, and therefore $A_2/A_1$  (the closure of $(\wt G^{000}\cap A)/A_1$ in $A/A_1$)  contains $(\wt G^{00}\cap A)/A_1$, which was to be shown.
\end{proof}

\section*{Acknowledgements}
We would like to thank Itay Kaplan for providing a short argument for a missing implication in Theorem~\ref{thm:mainB}.

\printbibliography
\end{document}